\documentclass[final, reqno, 11pt]{amsart}
\usepackage{amsmath}
\usepackage{amsfonts}
\usepackage{amssymb}
\usepackage{amsthm}

\newtheorem{theorem}{Theorem}[section]
\newtheorem{lemma}[theorem]{Lemma}
\newtheorem{proposition}[theorem]{Proposition}
\newtheorem{corollary}[theorem]{Corollary}
\theoremstyle{remark}

\newtheorem{definition}{Definition}[section]
\newcommand{\set}{\mathbb}
\newcommand{\dl}{\nabla}
\newcommand{\les}{\lesssim}
\renewcommand{\frak}{\mathfrak}

\newcommand{\mc}{\mathcal}
\newcommand{\be}{\begin{equation}}
\newcommand{\ee}{\end{equation}}
\newcommand{\bee}{\begin{align}}
\newcommand{\eee}{\end{align}}
\newcommand{\ba}{\begin{array}}
\newcommand{\ds}{\displaystyle}

\newcommand{\ea}{\end{array}}
\newcommand{\bpm}{\begin{pmatrix}}
\newcommand{\epm}{\end{pmatrix}}
\newcommand{\lb}{\label}

\DeclareMathOperator{\supp}{supp}

\DeclareMathOperator{\Imim}{Im}

\newcommand{\oast}{\circledast}
\DeclareMathOperator*{\slim}{s-lim}
\DeclareMathOperator*{\esssup}{ess\; sup}
\newcommand{\ov}{\overline}
\newcommand{\dd}{{\,}{d}}

\renewcommand{\Im}{\Imim}
\newcommand{\R}{\mathbb R}
\newcommand{\C}{\mathbb C}
\newcommand{\Z}{\mathbb Z}

\newcommand{\B}{\mathcal B}
\DeclareMathOperator{\ISO}{O}
\DeclareMathOperator{\ISSO}{ISO}

\title[$L^1$ Wave Operators]{Structure of Wave Operators for a Scaling-Critical Class of Potentials}
\author{Marius Beceanu}
\address{Rutgers University Department of Mathematics, 110 Frelinghuysen Rd., Piscataway, NJ, 08854, USA}
\email{mbeceanu@math.rutgers.edu}
\thanks{The author has been supported by the ANR program PREFERED}
\subjclass[2010]{35P25, 47A40, 81U05, 35J10, 35Q41, 35L05}
\begin{document}
\maketitle
\numberwithin{equation}{section}
\begin{abstract} We prove a structure formula for the wave operators in $\R^3$
$$
W_{\pm} = \slim _{t \to \pm \infty} e^{it (-\Delta + V)} P_c e^{it\Delta}
$$
and their adjoints for a scaling-invariant class of scalar potentials $V \in B$,
$$
B = \Big\{V | \sum_{k \in \set Z} 2^{k/2} \|\chi_{|x| \in [2^k, 2^{k+1}]}(x) V(x)\|_{L^2} < \infty\Big\},
$$
under the assumption that zero is neither an eigenvalue, nor a resonance for $-\Delta+V$.\\
The formula implies the boundedness of wave operators on $L^p$ spaces, $1 \leq p \leq \infty$, on weighted $L^p$ spaces, and on Sobolev spaces, as well as multilinear estimates for $e^{itH} P_c$.\\
When $V$ decreases rapidly at infinity, we obtain an asymptotic expansion of the wave operators. The first term of the expansion is of order $\langle y \rangle^{-4}$, commutes with the Laplacian, and exists when $V \in \langle x \rangle^{-3/2-\epsilon} L^{2, 1}$.\\
We also prove that the scattering operator $S = W_-^* W_+$ is an integrable combination of isometries.\\
The proof is based on an abstract version of Wiener's theorem, applied in a new function space.
\end{abstract}

\tableofcontents

\section{Introduction}
\subsection{Main result}
Let $V$ be a real-valued scalar potential in $\set R^3$. Consider the free Hamiltonian $H_0 = - \Delta$, the perturbed Hamiltonian $H = -\Delta+V$, and let $P_c$ be the projection on the continuous spectrum of $H$. Wave operators are defined by
\be\begin{aligned}\lb{eq1.1}
W_{\pm} := W_{\pm}(H, H_0) := \slim_{t \to \pm \infty} e^{it H} e^{-itH_0}.
\end{aligned}\ee
Also note that the adjoints of the wave operators are given by
$$
W^*_{\pm} = \slim_{t \to \pm \infty} e^{it H_0} e^{-itH} P_c.
$$
Consider potentials $V$ belonging to the scaling-critical Banach space
\be\lb{eq1.6}
B = \Big\{V \big| \sum_{k \in \set Z} 2^{k/2} \|\chi_{|x| \in [2^k, 2^{k+1}]}(x) V(x)\|_{L^2} < \infty\Big\}.
\ee

This choice is motivated by the fact that $B$ is the real interpolation space $B = (L^2, |x|^{-1} L^2)_{\frac 1 2, 1}$. For more details about real interpolation, see \cite{bergh}.

Note that $\langle x \rangle^{-1/2-\epsilon} L^2 \subset B \subset L^{3/2, 1}$. $B$ is critical with respect to the rescaling $(t, x) \mapsto (\alpha^2 t, \alpha x)$ in (\ref{eq1.1}). Consequently, replacing $V$ by $\alpha^{2} V(\alpha x)$, for some $\alpha > 0$, preserves $\|V\|_B$ and the estimate (\ref{1.9}), up to constants.

Let $\mc M$ be the space of finite mass Borel measures and $\mc M_{loc}$ be the space of Borel measures with locally finite mass. Also let $\ISO(3) = \{s \in \B(\R^3, \R^3) \mid s^* s = I\}$ be the group of orthogonal linear transformations (isometries) on~$\R^3$.

The main result of this paper is then expressed by Theorem \ref{theorem_1.1}.
\begin{theorem}\lb{theorem_1.1}
Assume that $V \in B$ is real-valued and that $H = -\Delta+V$ admits no eigenfunction or resonance at zero. Then for each of $W_{\pm}$ and $W_{\pm}^*$ there exists $g_{s, y}(x) \in (\mc M_{loc})_{s, y, x}$ such that $\|g_{y, s}(x)\|_{L^{\infty}_x} \in L^1_y \mc M_s$, i.e.
$$
\int_{\set R^3} \Big(\int_{\ISO(3)} \dd \|g_{s, y}\|_{L^{\infty}_x}\Big) \dd y < \infty
$$
and for $f \in L^2$ one has the representation formula
\be\begin{aligned}\lb{eqn1.8}
(W f)(x) &= f(x) + \int_{\R^3} \Big(\int_{\ISO(3)} \dd g_{s, y}(x) f(sx + y)\Big) \dd y.
\end{aligned}\ee
Here $W$ is any of $W_{\pm}$ and $W_{\pm}^*$.

Thus, $W_\pm$ and $W_\pm^*$ are bounded on $L^p$, $1 \leq p \leq \infty$: if $f \in L^2 \cap L^p$, then
\be\lb{1.9}
\|W_{\pm} f\|_{L^p} + \|W_{\pm}^* f\|_{L^p} \les \|f\|_{L^p}.
\ee
\end{theorem}

Recall that eigenfunctions at energy $\lambda$ are $L^2$ solutions of the equation $Hf = \lambda f$. Resonances are defined as solutions $f$ of $Hf = \lambda f$ that are not in $L^2$, but belong to $\langle x\rangle^\sigma L^2$ for all $\sigma>1/2$.


\noindent{\bfseries Interpretation.} Let elementary transformations be maps of the form $f(x) \mapsto g(x) f(sx + y)$, where $g \in L^{\infty}$, $s \in \ISO(3) := \{s \in \B(\R^3, \R^3) \mid s^*s = I\}$, and $y \in \R^3$.

Elementary transformations are composed of translations, multiplication by a bounded function, and orthogonal linear transformations.

Then (\ref{eqn1.8}) states that the wave operators $W_{\pm}$ given by (\ref{eq1.1}) and their adjoints $W_{\pm}^*$ are combinations of elementary transformations, integrable with respect to a measure on $\ISO(3) \times \R^3$ that is absolutely continuous in the $y \in \R^3$ variable, but singular in $s \in \ISO(3)$.


For $f \in L^2 \cap L^p$, $W_{\pm} f \in L^p$ by (\ref{1.9}).
Since $L^2 \cap L^p$ is dense in $L^p$ when $1\leq p<\infty$, $W_{\pm}$ and $W_{\pm}^*$ then admit unique $L^p$-bounded extensions --- and likewise on $L^{\infty}_0$, the $L^{\infty}$ closure of $L^2 \cap L^{\infty}$.

Formula (\ref{eqn1.8}) also gives a weakly continuous extension of $W_{\pm}$ to $L^p$ such that, if $f_n \rightharpoonup f$ weakly or weakly-$*$, then $W_{\pm} f_n \rightharpoonup W_{\pm} f$. This defines wave operators on non-separable Banach spaces such as $L^{\infty}$ or $\mc M$, in which $L^2$ is not dense, but whose dual or predual is separable.

Extensions and applications of Theorem \ref{theorem_1.1} are presented in Section \ref{sect_extapp}. They include a structure formula for the scattering operator, the boundedness of wave operators on Sobolev and weighted $L^p$ spaces, and a multilinear estimate that is a specific application of (\ref{eqn1.8}).

Finally, the proof also implies the norm continuity of the wave operators and of the coefficients $g_{s, y}$ from (\ref{eqn1.8}), as functions of the potential $V \in B$.

\subsection{Overview and history of the problem}
The notion of wave operators was introduced in the work of Moller and Friedrichs in the 1940s, then developed by Jauch, Cook, and Kato. For an account of these early findings and of the theory of wave operators in a Hilbert space setting, the reader is referred to Reed--Simon \cite{reesim}.

Wave operators are said to be \emph{asymptotically complete} in $L^2$ when
\begin{enumerate}
\item[i] $W_{\pm}$ are bounded and surjective from $L^2$ to $P_c L^2$.
\item[ii] The singular continuous spectrum of $H$ is empty, $\sigma_{sc}(H) = \varnothing$.
\end{enumerate}
In particular, this is the case when $V \in \langle x \rangle^{-1-\epsilon} L^{\infty}$, $\epsilon>0$. This fundamental result due to Agmon \cite{agmon} is based on and completes earlier work of Kato \cite{kato}, Kuroda, and others.

More recently, Ionescu--Schlag \cite{ionschlag} showed the asymptotic completeness of wave operators for potentials $V \in L^{3/2}$, $V \in L^2$, and in even more general classes, including some magnetic potentials, i.e.\ with gradient terms.

When the wave operators are asymptotically complete, they define a partial isometry between $L^2$ and $P_c L^2$, meaning that
\be
W_+^* W_+ = W_-^* W_- = I,\ W_+ W_+^* = W_- W_-^* = P_c.
\ee

Wave operators are useful in the theoretical study of  scattering. $W_{\pm}$ measure the similarity between the perturbed Schr\"{o}\-din\-ger evolution $e^{itH} P_c$ and the free evolution $e^{itH_0}$, being the identity when $V \equiv 0$.

Wave operators also help define the \emph{scattering operator} $S$, a fundamental notion in quantum mechanics, by
\be\lb{1.5}
S = W_-^* W_+.
\ee
$S$ and $S^*$ commute with $H_0$. Analogously one can define $\tilde S = W_- W_+^*$, which commutes with $H$. These notions were introduced in the work of Eckstein, Berezin--Fadeev--Minlos, and Jauch in the 1950s. For more details, we refer the reader to \cite{reesim}.

The scattering operator $S$ describes how a plane wave $e^{ix\xi}$ coming in from infinity scatters, upon encountering the potential $V$, into a superposition of plane waves as time goes to infinity.

The $L^1$ theory of wave operators is newer and has been developed by Yajima, beginning with his seminal paper \cite{yajima0} and with \cite{yajima}. 
A main application is transferring the dispersive properties of the free evolution $e^{itH_0}$ to the dispersive part of the perturbed evolution $e^{itH} P_c$. This is based on the intertwining property
\be\lb{1.4}
e^{itH} P_c = W_{\pm} e^{itH_0} W_{\pm}^*.
\ee

Thus, any linear estimates that hold for $e^{itH_0}$ carry over to $e^{itH} P_c$ when wave operators are bounded on the proper $L^p$ spaces. These include Strichartz inequalities, local smoothing estimates, and $L^p \to L^q$ decay estimates --- both for the Schr\"{o}dinger equation and for the wave equation.

Linear estimates carry over to the perturbed case because of the $L^p$ wave operator boundedness. In addition, some multilinear estimates, such as Proposition \ref{multi}, that hold for  $e^{itH_0}$ are transferred to $e^{itH}$ by the structure formula (\ref{eqn1.8}).

Some of these results admit more direct proofs, e.g.\ \cite{jss} and \cite{keetao}, and are more general than the $L^1$ boundedness of the wave operators. However, (\ref{1.4}) provides a straightforward proof when $W_\pm$ and $W_\pm^*$ are bounded.

Relation (\ref{1.4}) also gives rise to a functional calculus for $H$, as per \cite{yajima}. When $f \in L^{\infty}$, when $f$ is a Mihlin multiplier ($|\partial^{\alpha} f| \les_{\alpha} |x|^{-|\alpha|}$), or when $f \in \widehat {L^1}$, take in each case
$$
f(\sqrt{H P_c}) = W_{\pm} f(\sqrt{H_0}) W_{\pm}^*.
$$
The operators $f(\sqrt{H P_c})$ thus defined form commutative algebras of bounded $L^2$, $L^p$, $1<p<\infty$, and $L^1$ operators, respectively.

The first example is the usual functional calculus for selfadjoint operators, see \cite{reesim1}. The other two are specifically related to the $L^1$ boundedness of wave operators. One obtains Paley-Wiener projections, Littlewood-Paley square functions, and Sobolev and Besov spaces defined with respect to $H=-\Delta+V$ instead of $H_0=-\Delta$.

All previous results concerning the boundedness of wave operators on $L^p$ spaces, $p \ne 2$, are due to Yajima, beginning with his seminal paper \cite{yajima0} --- and to Artbazar--Yajima \cite{artyaj1d} and D'Ancona--Fanelli \cite{danfan1d} in one dimension. Since then, Yajima and his collaborators have obtained theorems that apply to all odd dimensions $d \geq 3$ \cite{yajima0} \cite{yajima} \cite{yajima5}, all even dimensions $d \geq 4$ \cite{yajima8} \cite{yajima7} \cite{yajima9}, as well as to the $2$-dimensional case \cite{yajima6}, \cite{jenyaj}.

One has to distinguish between operators without null eigenvalues or resonances --- of \emph{generic type}, see \cite{jenkat} --- and the situation when zero is an eigenvalue or a resonance for $H$ --- i.e.\ $H$ is of \emph{exceptional type} (the precise definition varies according to dimension). Yajima also made, for all dimensions $d \geq 3$, the assumption that
\be
\big(\langle x \rangle^{\frac{2(d-2)}{d-1}+\epsilon}V(x)\big)^{\wedge} \in L^{\frac{d-1}{d-2}}.
\ee
This becomes $\langle x \rangle^{1+\epsilon} V \in L^2$ in $\R^3$.
\begin{list}{\labelitemi}{\leftmargin=1.2em}
\item[1.] In $\R^3$ Yajima \cite{yajima} \cite{yajima5} proved the $L^p$ boundedness of $W_{\pm}$ for $1 \leq p \leq \infty$, provided that $|V(x)| \les \langle x \rangle^{-5-\epsilon}$ and $H$ is of generic type, and for $3/2 < p < 3$, if $|V(x)| \les \langle x \rangle^{-6-\epsilon}$ and $H$ is of exceptional type.
\item[2.] For odd $d \geq 5$, Yajima \cite{yajima} \cite{yajima5} obtained the $L^p$ boundedness of the wave operators in $\R^d$ for $1 \leq p \leq \infty$, if $|V(x)| \les \langle x \rangle^{-d-2-\epsilon}$ and $H$ is of generic type, and for $\frac d {d-2} < p < \frac d 2$, if $|V(x)| \les \langle x \rangle^{-d-3-\epsilon}$ and $H$ is of exceptional type.
\item[3.] For even $d \geq 6$, Finco--Yajima \cite{yajima9} showed the $L^p$ boundedness of the wave operators in $\R^d$:
\begin{list}{\labelitemi}{\leftmargin=1.2em}
\item[i] for $1 \leq p \leq \infty$, if $|V| \les \langle x \rangle^{-d-2-\epsilon}$ and $H$ is of generic type
\item[ii] for $\frac d {d-2} < p < \frac d 2$, if $H$ is of exceptional type and \begin{list}{\labelitemi}{\leftmargin=1.2em}
\item[a)]
$|V(x)| \les \langle x \rangle^{-d-3-\epsilon}$, $d \geq 8$, or
\item[b)]
$|V(x)| \les \langle x \rangle^{-10-\epsilon}$, $d = 6$.
\end{list}
\end{list}
\item[4.] In $\R^4$, Yajima \cite{yajima8} \cite{yajima7} showed the boundedness of the wave operators on $L^p$, $1 \leq p \leq \infty$, for $H$ of generic type, when
\begin{list}{\labelitemi}{\leftmargin=1.2em}
\item[i] $V \geq 0$ and $|D^{\alpha} V| \les \langle x \rangle^{-7-\epsilon}$ for all $|\alpha| \leq 4$ or
\item[ii] $\ds \sup_{x \in \set R^4} \langle x \rangle^{7+\epsilon} \Big(\int_{|x-y| \leq 1} |D^{\alpha} V(y)|^{2+\epsilon} \dd y \Big)^{\frac 1 {2+\epsilon}} < \infty$ for all $|\alpha| \leq 1$.
\end{list}
Yajima obtained similar conclusions in even dimensions $d \geq 6$, but the result of Finco--Yajima \cite{yajima9} supersedes them.
\item[5.] In $\R^2$ Jensen--Yajima \cite{jenyaj} showed that, if $V$ is of generic type and $|V(x)| \les \langle x\rangle^{-6-\epsilon}$, then $W_{\pm}$ are bounded in $L^p$, $1<p<\infty$.
\end{list}

In addition, for all dimensions $d\geq 3$, Yajima proved the boundedness of the wave operators if $\big\|\big(\langle x \rangle^{\frac{2(d-2)}{d-1}+\epsilon}V(x)\big)^{\wedge}\big\|_ {L^{\frac{d-1}{d-2}}}$ is sufficiently small. As noted, this becomes $\|\langle x \rangle^{1+\epsilon} V(x)\|_{L^2}$ in $\R^3$.

In this paper we start with an asymptotic expansion of the wave operator. For $f \in L^2$
\begin{align}\lb{2.2}
W_+ Z &= Z + W_{1+} Z + \ldots + W_{n+} Z + \ldots,\\
\nonumber W_{1+} Z &= i \int_{t>0} e^{-i t \Delta} V e^{i t \Delta} Z \dd t,\ \ldots \\
W_{n+} Z &=(-1)^{n-1} i^n \int_{t>s_1>\ldots>s_{n-1}>0} e^{-i(t-s_1)\Delta} V e^{-i(s_1-s_2) \Delta} V \ldots \\
\nonumber &e^{-i s_{n-1} \Delta} V e^{it\Delta} Z \dd t \dd s_1 \ldots \dd s_{n-1}.
\end{align}
We derive this expansion by Duhamel's identity (\ref{duham}) in Section \ref{lemmas}.

The first term is the identity, hence always bounded. Yajima proved in \cite{yajima0} that each remaining term $W_{n+}$, $n \geq 1$, is bounded as an $L^p$ operator, of norm that grows exponentially with $n$: in $\R^3$
\be\lb{1.7}
\|W_{n+} f\|_{L^p} \les C^n \|V\|_{\langle x \rangle^{-1-\epsilon} L^2}^n \|f\|_{L^p}.
\ee

Thus, as noted by Yajima \cite{yajima0}, when $\|V\|_{\langle x \rangle^{-1-\epsilon} L^2} << 1$ Weierstrass's criterion shows that (\ref{2.2}) is summable, hence $W_+$ is $L^p$-bounded. In general the asymptotic expansion (\ref{2.2}) may diverge.

In order to overcome this difficulty, for large $V$ Yajima \cite{yajima0} estimated a finite number of terms directly by this method. He used a separate computation to show the boundedness of the remainder, for which he had to assume that $V$ decays faster than~$\langle x \rangle^{-5-\epsilon}$.

Theorem \ref{theorem_1.1} proves the $L^1$ boundedness of wave operators for potentials $V$ in the scaling-invariant class $B$ defined by (\ref{eq1.6}). $B$ consists of $L^2$ functions weighted on dyadic shells, with a summability condition for the weights. $B$ is similar to the class $\langle x \rangle^{-1-\epsilon} L^2$ considered by Yajima, but requires less decay and is scaling-invariant. In addition, Theorem \ref{theorem_1.1} applies to arbitrarily large potentials and we obtain a structure formula, on top of $L^1$ boundedness.

We achieve this by using a summation method in (\ref{1.7}) --- an abstract version of Wiener's theorem, Theorem \ref{thm7} --- which also works for divergent asymptotic expansions, eliminating the need for a separate analysis of the remainder. This is the same method as in \cite{bec} or \cite{becgol}, applied in a different space of functions. However, in the current paper Wiener's theorem is completely intertwined with the rest of the proof, making its independent abstract formulation less useful.

Up to a point, the underlying computations parallel those of \cite{yajima0}. The spaces used in proving Theorem \ref{thm7} are new.

We require the absence of threshold eigenvalues or resonances from the continuous spectrum of $H = - \Delta + V$. Their presence leads to substantially different results, as shown by \cite{yajima5} and \cite{yajima9}.

Besides $L^p$ boundedness, Yajima \cite{yajima0} \cite{yajima} \cite{yajima7} \cite{yajima8} and Finco--Yajima \cite{yajima9} proved the boundedness of wave operators on Sobolev spaces $W^{\ell, p}$, in some cases for $1 \leq p \leq \infty$, in others without the endpoints $1$ and~$\infty$:
\begin{list}{\labelitemi}{\leftmargin=1.2em}

\item[1.] \cite{yajima8} showed that $W_\pm$ and $W_\pm^*$ are bounded in $\R^d$, $d \geq 3$, on $W^{\ell, p}$, $1 \leq p \leq \infty$, for $\|\partial^{\alpha} V(y)\|_{L^{d/2+\epsilon}(|y-x|<1)} \les \langle x \rangle^{-3d/2-1-\epsilon}$ for all $|\alpha| \leq \ell + \ell_0$, where $\ell_0 = 0$ when $d=3$ and $l_0 = \lfloor (d-1)/2 \rfloor$ when $d \geq 4$.

\item[2.] \cite{yajima9} showed that $W_\pm$ and $W_\pm^*$ are bounded in $\R^d$, $d \geq 3$ odd or $d \geq 6$ even, on $W^{\ell+2, p}$, $1 < p < \infty$, if $\big(\langle x \rangle^{\frac{2(d-2)}{d-1}+\epsilon} V(x)\big)^{\wedge} \in L^{\frac{d-1}{d-2}}$, $|V(x)| \les \langle x \rangle^{-d-2-\epsilon}$, and $\partial^\alpha V(x)$ are bounded for $|\alpha| \leq \ell$. For $p = 1$ and $p = \infty$, $W_\pm$ were shown to be bounded in $W^{k,p}(\R^d)$ if $\big(\langle x \rangle^{\frac{2(d-2)}{d-1}+\epsilon} \partial^{\alpha} V(x)\big)^{\wedge} \in L^{\frac{d-1}{d-2}}$ for all $|\alpha| \leq \ell$ and $|\partial^\alpha V (x)| \leq \langle x \rangle^{-d-2-\epsilon}$, $0 \leq \alpha \leq \ell$. 

\end{list}

These estimates require two fewer derivatives of the potential than the degree of regularity obtained for $W_\pm$, for $1<p<\infty$, and the same degree of regularity when $p=1$ or $p=\infty$.

In Corollary \ref{cor_sobolev} and in Corollary \ref{cor_sobolev_mare} we improve this result by one derivative in the endpoint case $p=1$. Assuming no regularity for $V$, we prove that $W_{\pm}$ are bounded on $\dot W^{1, 1}$.

\subsection{Further research directions} We conduct the study of wave and scattering operators in $\R^3$ for clarity, but the same method works in all higher dimensions.

The reader is referred to Theorem \ref{thm_general} and its proof, which unlike Theorem \ref{theorem_1.1} generalizes to all dimensions $d \geq 3$ for the potential space
$$
B_d := \Big\{V \mid \sum_{k \in \Z} 2^{\frac{d-2}{d-1} k} \|\chi_{|x| \in [2^k, 2^{k+1}]}(x) V(x)\|_{\widehat {L^{\frac {d-1} {d-2}}}} < \infty \Big\}.
$$
In this expression $\chi$ are smooth cutoff functions. This result will be the subject of a future paper.

Another case of interest is that of nonselfadjoint potentials obtained by linearizing Schr\"{o}dinger's equation around solitons. The main difference is the necessity of proving an extra estimate of the type
$$
\Big\|\int_0^{\infty} \widehat V(s\omega) e^{-st} \dd s \Big\|_{L^1_{t, \omega}} \les \|V\|_B.
$$
For simplicity, we only treat the selfadjoint case in this paper.

\subsection{Extensions and applications. The scattering operator}\lb{sect_extapp}

The first application to (\ref{eqn1.8}) is a structure formula for the scattering operator (\ref{1.5}).


\begin{theorem}\lb{thm_s} Assume that $V \in B$ is real-valued and zero is neither an eigenvalue, nor a resonance for $H=-\Delta+V$. Then $S$ is an integrable combination of isometries: for $f \in L^2$,
\be\begin{aligned}\lb{eqn1.11}
(Sf)(x) &= f(x) + \int_{\set R^3} \Big(\int_{\ISO(3)} f(sx + y) \dd g_{s, y}\Big) \dd y,\\
&\int_{\set R^3} \Big(\int_{\ISO(3)} \dd|g_{s, y}|\Big) \dd y < \infty.
\end{aligned}\ee
Consider any Banach space $A$ of functions on $\set R^3$ such that the norm of $A$ is invariant under isometries. Then, for any $f \in L^2 \cap A$,
\be
\|S f\|_A + \|S^* f \|_A \les \|f\|_A.
\ee
\end{theorem}

\noindent{\bfseries Interpretation.} The structure formula for wave operators (\ref{eqn1.8}) translates into a similar one (\ref{eqn1.11}) for the scattering operator $S$. The main difference is that, while $W_\pm$ need not commute with $-\Delta$, $S$ always does.

Although in general elementary transformations need not commute with $-\Delta$, we prove that $S-I$ is constituted only of elementary transformations of the form $f(x) \mapsto f(s x + y)$. These affine isometries commute with $-\Delta$.

$S$ being an integrable combination of isometries can be written $S \in \mc M(\ISSO(3))$, where $\mc M(\ISSO(3))$ is the space of finite-mass Borel measures on the group of affine isometries on $\R^3$. $\mc M(\ISSO(3))$ has a natural algebra structure.

In particular, $S$ is bounded on all Banach spaces of functions on $\R^3$ which are invariant under affine isometries --- Sobolev, Besov, Lipschitz, and Lorentz spaces included.

We further interpret (\ref{eqn1.11}) as stating that when plane waves come from infinity and encounter the potential $V$, some portion is transmitted and the rest is reflected with a phase shift. The reflected part has the same frequency $|\xi|$ and energy $E = |\xi|^2$ as the original, but a different direction and phase.\\

\paragraph{\bfseries Regularity of wave operators.} Although, unlike $S$,  $W_{\pm}$ need not be constituted of isometries, their constitutive elementary transformations are also bounded on many function~spaces.
\begin{theorem}\lb{thm_general} Consider a Banach space $A$ of functions on $\set R^3$ such that the norm of $A$ is invariant under isometries and for any $\omega \in S^2$
\be\lb{1.12}
\|\chi_{\{x \mid x \cdot \omega \geq 0\}}(x) f(x)\|_A \les \|f\|_A.
\ee
If $V \in B$ is real-valued and zero is neither an eigenvalue, nor a resonance for $H=-\Delta+V$, then for every $f \in L^2 \cap A$
\be
\|W_{\pm} f\|_A + \|W_{\pm}^* f \|_A \les \|f\|_A.
\ee
\end{theorem}


This applies, directly or indirectly, to several Sobolev or Besov spaces, such as
$$
\dot W^{1, 1} = \{f \mid |\dl f| \in L^1\},\ \dot W^{1, \mc M} = \{f \mid |\dl f| \in \mc M\},
$$
as well as the space $C_b$ of continuous bounded functions:
\begin{corollary}\lb{cor_sobolev} Let $V \in B$ be real-valued, such that zero is neither an eigenvalue, nor a resonance for $H$. Then $W_{\pm}$ and $W_{\pm}^*$ are bounded on $\dot W^{1, 1}$, on $\dot W^{1, \mc M}$, on $\dot H^s$, $s \in [0, 1/2)$, on $\dot W^{s, p}$, $s \in [0, 1/p)$, for $1<p<\infty$, and on $C_b$.
\end{corollary}

We also obtain that, given enough regularity of $V$, $W_\pm$ and $W_\pm^*$ are bounded on Sobolev spaces of any~order.
\begin{corollary}\lb{cor_sobolev_mare} Let $V \in B$ be real-valued, such that zero is neither an eigenvalue, nor a resonance for $H$. For $s \in \R$ and $1<p<\infty$, assume that
\be\lb{cond_V}
V \in \left\{\begin{aligned}
&B, && |s|<2,\ sp<3 \\
&B \cap L^{3/2+\epsilon},\ && |s|<2, sp=3 \\
&B \cap L^p, && |s|<2, sp>3 \\
&B \cap W^{|s|-2, 1} \cap W^{|s|-2, p+\epsilon}, && |s| \geq 2.
\end{aligned}\right.
\ee
Then $W_{\pm}$ and $W_{\pm}^*$ are bounded on $W^{s, p}$.
\end{corollary}
The proof of Corollary \ref{cor_sobolev_mare} uses the intertwining property (\ref{1.4}) and elliptic regularity for the free Hamiltonian $H_0=-\Delta$. One can further relax (\ref{cond_V}). The reader is directed to proof of Corollary \ref{cor_sobolev_mare} for sharper conditions under which this conclusion holds.\\

\paragraph{\bfseries Weighted spaces.} $W_{\pm}$ are bounded on weighted $L^p$ spaces, with weights up to $\langle x \rangle^{-1+\epsilon}$. Due to the introduction of inhomogenous weights, this result is no longer scaling-invariant.
\begin{corollary}\lb{wei}
Let $V$ be real-valued with zero being neither an eigenvalue, nor a resonance for $H=-\Delta+V$. Assume that $V \in \langle x \rangle^{-\alpha} L^2$, $1/2< \alpha<3/2$. Then $W_{\pm}$ and $W_{\pm}^*$ are bounded on $\langle x \rangle^\beta L^p$, $1 \leq p \leq \infty$, for $|\beta| < \alpha-1/2$.
\end{corollary}


One can also obtain wave operator estimates on weighted Sobolev spaces.\\

\paragraph{\bfseries Asymptotic expansion.} Besides results that hold under scaling-invariant conditions on $V$, we retrieve an asymptotic expansion of the wave operators $W_\pm$ when $V$ decays rapidly.
\begin{proposition}\lb{asimptotic}
Assume that $V \in \langle x \rangle^{-3/2-\epsilon} L^2$ and that $H$ has no resonance or eigenstate at zero. Then, for $g_{s, y}(x)$ as in (\ref{eqn1.8}), there exists $g^1_{s, y}(x) \equiv g^1_{s, y}$ constant in $x$ such that
$$
g_{s, y}(x)-g^1_{s, y}(x) \in \langle y \rangle^{-1-\epsilon} L^1_y \mc M_s L^{\infty}_x,\ g^1_{s, y}(x) \in \langle y \rangle^{-4} L^{\infty}_y \mc M_s L^{\infty}_x.
$$
\end{proposition}
The leading-order term is given by $g^1_{s, y}(x)$ and has size $\langle y \rangle^{-4}$. It has better properties than subsequent terms: it is constant in the $x$ variable, hence it commutes with~$-\Delta$.

Given enough decay, this expansion can be continued to any order. For example, if $V \in \langle x \rangle^{-5/2-\epsilon} L^2$, we get a second term in the asymptotic expansion, of size $\langle y \rangle^{-5}$, and a remainder in $\langle y \rangle^{-2-\epsilon} L^1_y$.

Analogous decay results hold for the scattering operator $S$. However, the explicit formula (\ref{2.152}) shows no direct connection between the leading-order terms of $S$ and of $W_+$.\\

\paragraph{\bfseries Multilinear estimates.} The structure of wave operators established in (\ref{eqn1.8}) intervenes in the proof of multilinear estimates.
\begin{proposition}\lb{multi} Assume that $V \in B$ is real-valued and that zero is neither an eigenvalue, nor a resonance for $H=-\Delta+V$; take $U \in L^{\infty} \cap L^{3/2}$. Then
\be\lb{multilin}
\Big|\int_0^{\infty} U (e^{it H} P_c f)^2 \dd t\Big| \les \|f\|_{H^{-1}}^2 \|U\|_{L^{\infty} \cap L^{3/2}}.
\ee
\end{proposition}
This multilinear estimate is not implied by the $L^1$ boundedness of wave operators. The proof uses the structure formula (\ref{eqn1.8}) of Theorem \ref{theorem_1.1}.

\section{Proof of the statements}
\subsection{Notations}

Let $R_0(\lambda) = (H_0 - \lambda)^{-1}$ be the free resolvent and let $R_V(\lambda) = (H-\lambda)^{-1}$ be the perturbed resolvent. Explicitly, in three dimensions and for $\Im \lambda \geq 0$, 
\be\lb{eq_3.55}
R_0(\lambda^2)(x, y) = \frac 1 {4\pi} \frac {e^{i \lambda |x-y|}}{|x-y|}.
\ee
$R_0(\lambda)$ is analytic in $\C \setminus [0, \infty)$ and discontinuous along $[0, \infty)$, taking different boundary values in the upper and in the lower half-plane; zero is a branching point. $R_V(\lambda)$ is meromorphic in $\C \setminus [0, \infty)$, with poles being the eigenvalues~of~$H$.


We also denote, for $\lambda \in \set R$,
\be\begin{aligned}\lb{eq_2.2}
R_{0a}(\lambda) &= \frac 1 i \big(R_0(\lambda+i0) - R_0(\lambda-i0)\big) \\
R_{Va}(\lambda) &= \frac 1 i \big(R_V(\lambda+i0) - R_V(\lambda-i0)\big),
\end{aligned}\ee
only when $\lambda \in \set R$ is not an eigenvalue of $H$, for the latter.\\

\paragraph{\bfseries{Other notations}}
We denote Lorenz spaces by $L^{p, q}$, $1\leq p, q \leq \infty$, Sobolev spaces by $W^{s, p}$, $s \in \set R$, $1 \leq p \leq \infty$, and fix the Fourier transform to
$$
\widehat f(\eta) = \int_{\set R^d} e^{-ix \eta} f(x) \dd x,\ f^{\vee}(x) = (2\pi)^{-d} \int_{\set R^d} e^{i\eta x} f(\eta) \dd \eta.
$$
We adopt the point of view according to which
$$\begin{aligned}
e^{itH_0} &= (R_{0a}(\lambda))^{\vee}(t);\ R_{0a}(\lambda) = (e^{itH_0})^{\wedge},\ \lambda \in \R;\\
i R_0(\lambda) &= (\chi_{[0, \infty)}(t) e^{itH_0})^{\wedge}(\lambda),\ \Im \lambda < 0.
\end{aligned}$$
We use the kernel $K(x_0, x_1)$ to represent an operator $K$ if
$$
\langle Kf, g \rangle = \int_{\R^6} f(x_0) K(x_0, x_1) \ov g(x_1) \dd x_0 \dd x_1.
$$
Composition of operators translates into that of the corresponding kernels:
\be\lb{comp_op}
(K_2 \circ K_1 f)(x_2) = \int_{\R^3} \Big(\int_{\R^3} K_1(x_0, x_1) K_2(x_1, x_2) \dd x_1\Big) f(x_0) \dd x_0.
\ee
Also, let
\begin{list}{\labelitemi}{\leftmargin=1em}
\item[$\ast$] $\chi_A$ be the characteristic function of the set $A$;
\item[$\ast$] $\mc M$ be the space of finite-mass Borel measures on $\set R$ or $\set R^d$;
\item[$\ast$] $\dot W^{1, \mc M}$ be the set of functions $f$ such that $|\dl f| \in \mc M$;
\item[$\ast$] $\delta_x$ denote Dirac's measure at $x$;
\item[$\ast$] $\langle x \rangle = (1+|x|^2)^{1/2}$;
\item[$\ast$] $\langle x \rangle^{\alpha} L^p = \{\langle x \rangle^{\alpha} f(x) \mid f \in L^p\}$ with the natural norm;
\item[$\ast$] $\B(B_1, B_2)$ be the Banach space of bounded operators from $B_1$ to $B_2$;
\item[$\ast$] $\mc F_{x_1, x_2, \ldots} f$ be the Fourier transform in $x_1$, $x_2$, etc., of $f$;
\item[$\ast$] $\widehat f$ be the Fourier transform of $f$ in all its variables;
\item[$\ast$] Fourier multipliers be $F(\dl) f := \big(F(i\xi) \widehat f(\xi)\big)^{\wedge}$;
\item[$\ast$] $C$ be any constant (not always the same throughout the paper);
\item[$\ast$] $\mc S$ be the Schwartz space;
\item[$\ast$] $S^{d-1}$ be the $d-1$-dimensional unit sphere;
\item[$\ast$] $S_{\omega}(x) = x - 2(x \cdot \omega) \omega$ be the reflection of $x$ along $\omega$;
\item[$\ast$] $ISO(d)$ be the group of isometries of $\set R^d$.
\end{list}

\subsection{Preliminary lemmas}\lb{lemmas}
We first explain the spectral condition of Theorem \ref{theorem_1.1}. The natural condition arising from the proof is that the equation
\be\lb{1.10}
f = - R_0(0) V f
\ee
has no solution $f \in L^{\infty}$. We prove this can be replaced with the condition stated in Theorem \ref{theorem_1.1} --- that there are no eigenvalues or resonances at zero.

The kernel of $R_0(0) = (-\Delta)^{-1}$  is explicitly given by (\ref{eq_3.55}), which in this case has the form $R_0(0)(x, y) = \frac 1 {4\pi} \frac 1 {|x-y|}$.

\begin{lemma}\lb{lem2.1} Assume that $V \in L^{3/2, 1}$ and let $H=-\Delta + V$. If (\ref{1.10}) has a solution $f \in L^{\infty}$, then $Hf = 0$ and $f \in L^{3, \infty} \subset B' \subset \langle x \rangle^\sigma L^2$ for any $\sigma > 1/2$.

Moreover, when $V \in B$, assume there exists a solution $f \in B'$ to (\ref{1.10}), i.e.\ one such that
$$
\sup_{k \in \Z} 2^{-k/2} \|\chi_{|x| \in [2^k, 2^{k+1}]}(x) f(x)\|_{L^2} < \infty.
$$
Then $f \in L^{\infty} \cap L^{3, \infty}$.
\end{lemma}
Here $L^{3, \infty}$ is the same as weak-$L^3$. The second part of the lemma shows the equivalence of the two notions of resonance, one involving $L^{\infty}$ and the other involving weighted $L^2$ spaces, i.e.\ $\langle x \rangle^{\sigma} L^2$ for any $\sigma > 1/2$.

\begin{proof} If such $f \in L^{\infty}$ exists, 
then $\partial_k \partial_\ell f$ are in $L^{3/2, 1}$ as well because the Riesz transform is bounded on $L^{3/2, 1}$.

Thus $-\Delta f \in L^{3/2, 1}$ satisfies the equation $-\Delta f = \Delta R_0(0) V f = -V f$, so $H f = 0$.

Write $V = V_1 + V_2$, where $V_1$ is bounded of compact support and $\|V_2\|_{L^{3/2, 1}}$ is small. Then $(I + R_0(0) V_2)^{-1}$ is a bounded operator on $L^{3, \infty}$ and
$$
f = (I + R_0(0) V_2)^{-1} R_0(0) V_1 f.
$$
Since $V_1 f \in L^1$, $R_0(0) V_1 f \in L^{3, \infty}$, so $f \in L^{3, \infty} \subset B'$.

Hence $f$ belongs to $\langle x \rangle^{\sigma} L^2$ for any $\sigma>1/2$.

When $V \in B$ and $f \in B'$, then $V f \in L^1$, so $f = R_0(0) V f \in L^{3, \infty}$. We again split $V$ into two parts, $V = V_1 + V_2$, and write
$$
f = (I + R_0(0) V_2)^{-1} R_0(0) V_1 f.
$$
$(I + R_0(0) V_2)^{-1}$ is bounded on $L^{\infty}$ and $V_1 f \in L^{3/2, 1}$, $R_0(0) V_1 f \in L^{\infty}$, so $f \in L^{\infty}$.
\end{proof}

The discussion in the sequel is carried for $V \in L^{3/2, 1}$, because $B \subset L^{3/2, 1}$.



Consider $V \in L^{3/2, 1}$ and real-valued; also assume that the spectral condition (\ref{1.10}) holds. The spectral projection $P_p = I - P_c$, corresponding to the point spectrum, is a finite-rank operator of the form
\be
P_p = \sum_{\ell=1}^N \langle \cdot, f_{\ell} \rangle f_{\ell}.
\ee
$P_p$ is an orthogonal projection, i.e.\ $\langle f_{\ell_1}, f_{\ell_2} \rangle = \delta_{\ell_1}(\ell_2)$. $f_{\ell}$ are eigenfunctions of $H=-\Delta+V$: $H f_\ell = \lambda_\ell f_\ell$.

If the spectral condition (\ref{1.10}) holds, then $f_{\ell}$ decay exponentially by Agmon's bound, see \cite{agmon} and \cite{bec}, and there is no zero resonance, which would not be detected by $P_p$ anyway.

Next, we derive an elementary formula for the wave operators.
\begin{lemma} Assume that $V \in L^{3/2, 1}$ is real-valued and $H=-\Delta+V$ has no zero eigenvectors or resonances. Let $P_p=I-P_c$ be the projection on the point spectrum of $H$. Then $W_+$ exists, $W_+ = \slim_{t \to \infty} e^{itH} P_c e^{-itH_0}$, $W_+$ is $L^2$-bounded, and
\be\begin{aligned}\lb{eqn1.6}
W_+ f &= f + i \int_0^{\infty} e^{itH} V e^{-itH_0} f \dd t \\
&= P_c f + i \int_0^{\infty} e^{itH} P_c V e^{-itH_0} f \dd t.
\end{aligned}\ee
\end{lemma}
This lemma is a direct consequence of \cite{bec} and also holds when $V$ is in $L^{3/2, \infty}_0$. 
There exist similar formulae for $W_-$ and $W_{\pm}^*$; in particular,
\be\lb{eqn1.7}
W_-^* f = f + i \int_{-\infty}^0 e^{itH_0} V e^{-itH} f \dd t.
\ee
\begin{proof}
The second half of (\ref{eqn1.6}) is an immediate consequence of the endpoint Strichartz estimates of Keel--Tao \cite{keetao}, which hold in this case for both $e^{itH_0}$ and $e^{itH} P_c$; see \cite{bec} for the latter. Indeed, one has
$$\begin{aligned}
\|e^{it H_0} f\|_{L^2_t L^{6, 2}_x} &\les \|f\|_{L^2}\\
\Big\|\int e^{-isH_0} F(s) \dd s\Big\|_{L^2_x} &\les\|F\|_{L^2_t L^{6/5, 2}_x}, 
\end{aligned}$$
and likewise for $e^{it H} P_c$.

Recall that by Duhamel's formula
\be
e^{itH} P_c e^{-itH_0} f = P_c f + i \int_0^t e^{isH} P_c V e^{-isH_0} f \dd s.
\ee
Since this integral converges in norm, we let $t \to \infty$ and obtain the second part of (\ref{eqn1.6}). Thus, endpoint Strichartz estimates imply the existence of the strong limit $\slim_{t \to \infty} e^{-itH} P_c e^{itH_0}$ in $L^2$.

In the absence of zero eigenvectors or resonances, all eigenvectors, if they exist, must decay exponentially by Agmon's bound.
Then, for $f \in L^2$
\be\lb{28}
\lim_{t \to \infty} P_p e^{-itH_0} f = 0.
\ee
To show this, we approximate $f$ in $L^2$ by $L^1$ functions, for which $e^{itH_0} f$ decays pointwise like $|t|^{-3/2}$. Consequently,
$$
W_+ = \slim_{t \to \infty} e^{itH} e^{-itH_0} = \slim_{t \to \infty} e^{itH} P_c e^{-itH_0}.
$$

In order to pass to the first part of (\ref{eqn1.6}), a separate computation shows~that
$$
\slim_{t \to \infty} i \int_0^t e^{isH} P_p V e^{-isH_0} \dd s = \slim_{t \to \infty} P_p (I - e^{itH} e^{-itH_0}) = P_p.
$$
Indeed, for each eigenfunction $f_{\ell}$ corresponding to an eigenvalue $\lambda_{\ell}$
$$
f_{\ell} = -R_0(\lambda_{\ell}) V f_{\ell},
$$
so
$$
\slim_{t \to \infty} i \int_0^t e^{is\lambda_{\ell}} \langle V e^{-isH_0} f, f_{\ell} \rangle f_{\ell} \dd s = \slim_{t \to \infty} \langle (I - e^{it\lambda_{\ell}} e^{-itH_0}) f, f_{\ell} \rangle f_{\ell} = \langle f, f_{\ell} \rangle f_{\ell}.
$$
\end{proof}

Next, we derive the asymptotic expansion (\ref{2.2}) for $W_+$ by iterated applications of Duhamel's formula. For $f \in L^2$
\be\lb{duham}
e^{itH} f = e^{itH_0} f -i \int_0^t e^{i(t-s)H_0} V e^{isH} f \dd s.
\ee
Applying this in (\ref{eqn1.6}), we obtain (\ref{2.2}).

Since $W_{n+}$ and $W_{\pm}$ involve singular integrals akin to the Hilbert transform, for $\epsilon>0$ we introduce the mollified versions
\be\begin{aligned}\lb{eq2.5}
W_{n+}^\epsilon f&:= i^n \int_{0\leq t_1 \leq \ldots \leq t_n} e^{i(t_n-t_{n-1})H_0-\epsilon(t_n-t_{n-1})} V \ldots \\
& e^{i(t_2 - t_1)H_0 - \epsilon(t_2 - t_1)} V e^{it_1 H_0 - \epsilon t_1} V e^{-it_nH_0} f \dd t_1 \ldots \dd t_n,
\end{aligned}\ee
together with
\be\lb{eq2.6}
W_+^\epsilon = I + i \int_0^{\infty} e^{it H-\epsilon t} V e^{-it H_0} \dd t.
\ee
By (\ref{eqn1.6}) and (\ref{eq2.6}), $W_+^\epsilon f \to W_+ f$ as $\epsilon \to 0$ for each $f \in L^2$. Indeed, the dispersive terms converge by dominated convergence and for every $f \in L^2$
$$
\lim_{\epsilon \to 0} \int_0^{\infty} e^{-itH - \epsilon t} P_p e^{itH_0} \dd t = \int_0^{\infty} e^{-itH} P_p e^{itH_0} \dd t,
$$
where the left-hand side is absolutely integrable for each $\epsilon>0$ and the right-hand side is an improper integral.



%



In the sequel, we shall employ the following form of the operators $W_{n+}^\epsilon$ and $W_+^\epsilon$, introduced by Yajima in \cite{yajima0}:
\begin{definition} For $\epsilon \geq 0$, let $T_{1+}^\epsilon(x_0, x_1, y)$ be defined by
\be\lb{2.16}
(\mc F_{x_0, x_1, y} T_{1+}^\epsilon)(\xi_0, \xi_1, \eta) := \frac {\widehat V(\xi_1 - \xi_0)}{|\xi_1 + \eta|^2 - |\eta|^2 - i\epsilon}
\ee
and, more generally,
\be\begin{aligned}\lb{2.17}
(\mc F_{x_0, x_n, y} T_{n+}^\epsilon)(\xi_0, \xi_n, \eta) &:= 
\int_{\R^{3(n-1)}} \frac{\prod_{\ell=1}^n \widehat V(\xi_{\ell} - \xi_{\ell-1}) \dd \xi_1 \ldots \dd \xi_{n-1}}{\prod_{\ell=1}^n (|\xi_{\ell}+\eta|^2-|\eta|^2 - i \epsilon)}.
\end{aligned}\ee
Also let $T_+^\epsilon$ be given by
\be\begin{aligned}\lb{eqn2.13}
\mc F_{y} T_+^\epsilon(x_0, x_1, \eta) &:= e^{ix_0 \eta} \big(R_V(|\eta|^2 + i\epsilon) V\big)(x_0, x_1) e^{-ix_1\eta};
\end{aligned}\ee
see (\ref{2.24}) for an alternate form. Finally, define $T_{n-}^\epsilon$ and $T_-^\epsilon$ by
\be\begin{aligned}
(\mc F_{x_0, x_n, y} T_{n-}^\epsilon)(\xi_0, \xi_n, \eta) &:= 
\int_{\R^{3(n-1)}} \frac{\prod_{\ell=1}^n \widehat V(\xi_{\ell} - \xi_{\ell-1}) \dd \xi_1 \ldots \dd \xi_{n-1}}{\prod_{\ell=1}^n (|\xi_{\ell}+\eta|^2-|\eta|^2 + i \epsilon)},\\
\mc F_{y} T_-^\epsilon(x_0, x_1, \eta) &:= e^{ix_0 \eta} \big(R_V(|\eta|^2 - i\epsilon) V\big)(x_0, x_1) e^{-ix_1\eta}.
\end{aligned}\ee
\end{definition}

For each $\eta \in \R^3$, $\mc F_\eta T_+^\epsilon \in \B(L^1_{x_1}, L^1_{x_0})$ and $\mc F_\eta T_+^\epsilon \in \B(L^{\infty}_{x_0}, L^{\infty}_{x_1})$. The same holds for all the other kernels that we consider --- see Lemma \ref{lema2.4}.


\begin{lemma}\lb{lemma2.1} For $\epsilon>0$, $n \geq 1$,
\be\begin{aligned}\lb{2.19}
\langle W_{n+}^\epsilon f, g \rangle &= - \int_{\set R^9} T_{n+}^\epsilon(x_0, x, y) f(x-y) \ov g(x) \dd x_0 \dd y \dd x
\end{aligned}\ee
and
\be\begin{aligned}\lb{2.20}
\langle W_+^\epsilon f, g \rangle &= \int_{\set R^9} T_+^\epsilon(x_0, x, y) f(x-y) \ov g(x) \dd x_0 \dd y \dd x.
\end{aligned}\ee
\end{lemma}
We may as well assume for now that $V \in \mc S$, so that all the computations are easily justified. Still, these integrals become singular for $\epsilon = 0$, this being the reason why we introduced the parameter $\epsilon>0$.

All our conclusions concerning $T_+^\epsilon$ apply equally to $T_-^\epsilon$.

\begin{proof}
Firstly, note that
$$\begin{aligned}
\langle W_{1+}^\epsilon f, g\rangle &= i\int_0^{\infty} e^{it|\eta_1|^2-\epsilon t} \widehat V(\eta_1 - \eta_0) e^{-it|\eta_0|^2} \widehat f(\eta_0) \ov{\widehat g}(\eta_1) \dd \eta_1 \dd \eta_0 \dd t \\
&= \int_{\R^6} \frac {\widehat V(\eta_1 - \eta_0)} {|\eta_0|^2 - |\eta_1|^2 - i \epsilon} \widehat f(\eta_0) \ov {\widehat g}(\eta_1) \dd \eta_1 \dd \eta_0.
\end{aligned}
$$
Then, by redenoting $\eta_0 = \eta$,  $\eta_1 - \eta_0 = \xi$, we obtain exactly
\be\lb{eqn2.6}
\langle W_{1+}^\epsilon f, g\rangle = - \int_{\R^6} \frac {\widehat V(\xi)}{|\eta+\xi|^2- |\eta|^2 - i \epsilon} \widehat f(\eta) \ov{\widehat g}(\eta+\xi) \dd \eta \dd \xi.
\ee
More generally, when $n \geq 1$, the expression is given by
$$\begin{aligned}
\langle W_{n+}^\epsilon f, g\rangle &= (-1)^{n-1} i^n\int_{0=t_0\leq t_1\leq \ldots\leq t_n} \prod_{\ell = 1}^n \Big(e^{i(t_{\ell} - t_{\ell-1})|\eta_{\ell}|^2-(t_{\ell}-t_{\ell-1})\epsilon} \widehat V(\eta_{\ell}-\eta_{\ell-1}) \Big) \\
&e^{-it_n|\eta_0|^2} \cdot \widehat f(\eta_0) \ov{\widehat g}(\eta_n) \dd \eta_0 \ldots \dd \eta_n \dd t_1 \ldots \dd t_n,
\end{aligned}$$
where $t_0 = 0$. After integrating in $t_1, \ldots, t_n$, we obtain
$$\begin{aligned}
\langle W_{n+}^\epsilon f, g\rangle &= -\int \frac {\prod_{\ell=1}^n \widehat V(\eta_{\ell}-\eta_{\ell-1})} {\prod_{\ell=1}^n (|\eta_{\ell}|^2 - |\eta_0|^2 - i \epsilon)} \widehat f(\eta_0) \ov {\widehat g}(\eta_n) \dd \eta_0 \ldots \dd \eta_n.
\end{aligned}$$
Redenoting $\eta_0 = \eta$,  $\eta_{\ell} -\eta_0 = \xi_{\ell}$ leads to (\ref{eqn2.7}), for $\xi_0 = 0$:
\be\begin{aligned}\lb{eqn2.7}
\langle W_{n+}^\epsilon f, g\rangle = -\int \frac{\prod_{\ell=1}^n \widehat V(\xi_{\ell} - \xi_{\ell-1}) \dd \xi_1 \ldots \dd \xi_{n-1}}{\prod_{\ell=1}^n (|\eta+\xi_{\ell}|^2- |\eta|^2 - i \epsilon)} \widehat f(\eta) \ov{\widehat g}(\eta+\xi_n) \dd \eta \dd \xi_n.
\end{aligned}\ee
Then
$$\begin{aligned}
\langle W_{n+}^\epsilon f, g \rangle &= - \int_{\set R^6} \mc F_{x_0, x_n, y} T_{n+}^\epsilon(0, \xi_n, \eta) \widehat f(\eta) \ov {\widehat g}(\eta + \xi_n) \dd \eta \dd \xi_n \\
&= - \int_{\set R^9} T_{n+}^\epsilon(x_0, x, y) f(x-y) \ov g(x) \dd x_0 \dd y \dd x.
\end{aligned}$$
We can also write these operators as
$$
\mc F_{x_0, x_1, y} T_{1+}^\epsilon(\xi_0, \xi_1, \eta) = \mc F_{a, b} (R_0(|\eta|^2 + i\epsilon) V)(\xi_0+\eta, \xi_1+\eta),
$$
where $\mc F_{a, b}$ denotes the Fourier transform with respect to the two variables, and, in general,
$$
\mc F_{x_0, x_1, y} T_{n+}^\epsilon(\xi_0, \xi_1, \eta) = \mc F_{a, b} \big((R_0(|\eta|^2 + i\epsilon) V)^n\big)(\xi_0+\eta, \xi_n+\eta).
$$
Concerning the wave operator, we start from
$$\begin{aligned}
\langle W_+^\epsilon f, g\rangle &= I + i\int_0^{\infty} \int_{\R^6} \mc F_{a, b} \big(e^{itH-t\epsilon}V\big)(\eta_0, \eta_1) e^{-it|\eta_0|^2} \widehat f(\eta_0) \ov{\widehat g}(\eta_1) \dd \eta_1 \dd \eta_0 \dd t \\
&= I - \int_{\R^6} \mc F_{a, b} \big(R_V(|\eta_0|^2 + i\epsilon) V\big)(\eta_0, \eta_1) \widehat f(\eta_0) \ov{\widehat g}(\eta_1) \dd \eta_1 \dd \eta_0 \\
&= I - \int_{\R^6} \mc F_{a, b} (R_V(|\eta|^2 + i\epsilon) V)(\eta, \eta+\xi) \widehat f(\eta) \ov {\widehat g}(\eta + \xi) \dd \eta \dd \xi.
\end{aligned}
$$
For $T_+^\epsilon$ defined by (\ref{eqn2.13}), note that
\be\begin{aligned}\lb{2.24}
\mc F_{x_0, x_1, y} T_+^\epsilon(\xi_0, \xi_1, \eta) &= \mc F_{a, b} (R_V(|\eta|^2 + i\epsilon) V)(\xi_0+\eta, \xi_1+\eta).
\end{aligned}\ee
Then
$$\begin{aligned}\lb{eqn2.29}
\langle W_+^\epsilon f, g \rangle &= \langle f, g \rangle - \int_{\set R^6} \mc F_{x_0, x_1, y} T_+^\epsilon(0, \xi_1, \eta) \widehat f(\eta) \ov {\widehat g}(\eta + \xi_1) \dd \eta \dd \xi_1 \\
&= \langle f, g \rangle - \int_{\set R^9} T_+^\epsilon(x_0, x, y) f(x-y) \ov g(x) \dd x_0 \dd y \dd x.
\end{aligned}$$
\end{proof}


By (\ref{eqn2.6}), the kernel associated to $W_{1+}$ is $T_{1+}$, such that
$$
(\mc F_{x_0, x, y} T_{1+})(0, \xi, \eta) = \lim_{\epsilon \downarrow 0} (\mc F_{x_0, x, y} T_{1+}^\epsilon)(0, \xi, \eta) = \lim_{\epsilon \downarrow 0} \frac {\widehat V(\xi)}{|\eta+\xi|^2 - |\eta|^2 - i \epsilon}.
$$
Integrating in $x_0$, following (\ref{2.19}) and (\ref{2.20}), corresponds exactly to setting $\xi_0 = 0$ in (\ref{2.17}).

Two variables are entirely sufficient for representing $W_{n+}^\epsilon$, but we need one more variable for a meaningful algebra structure. This is the reason for having a third variable $x_0$ in (\ref{2.16}) and (\ref{2.17}).

For three-variable kernels $T(x_0, x_1, y)$ we define the composition law
\begin{definition}[Composition law $\oast$]
\be\begin{aligned}\lb{eq2.21}
(T_1 \oast T_2)(x_0, x_2, y) = \int_{\set R^6} T_1(x_0, x_1, y_1) T_2(x_1, x_2, y-y_1) \dd x_1 \dd y_1.
\end{aligned}\ee
\end{definition}
This law consists in convolution in the $y$ variable --- i.e.\ multiplication in the dual variable $\eta$ --- and composition of operators (\ref{comp_op}) in the other two.

Note that
$$
T_{m+}^\epsilon \oast T_{n+}^\epsilon = T_{(m+n)+}^\epsilon.
$$
Thus $\oast$ can be used to recursively generate all of $W_{n+}^\epsilon$, $n \geq 1$, starting from $W_{1+}^\epsilon$. Hence $\oast$ is the proper composition law to consider for these three-variable kernels. 
For example,
$$
\mc F_{x_0, x_2, y} T_{2+}(\xi_0, \xi_2, \eta) = \int_{\set R^3} \frac {\widehat V(\xi_2-\xi_1)}{|\xi_2 + \eta|^2 - |\eta|^2-i0} \cdot \frac {\widehat V(\xi_1-\xi_0)}{|\xi_1 + \eta|^2 - |\eta|^2-i0} \dd \xi_1.
$$

We write the resolvent identity as
\be\lb{2.29}
(I + R_0 V)^{-1} = I - R_V V;\ R_V = (I + R_0 V)^{-1} R_0.
\ee
Both sides of this equality --- also see (\ref{eqn2.13}) and (\ref{2.90}) --- are characterized by this lemma:
\begin{lemma}\lb{lemma2.3} Assume that $V \in L^{3/2, 1}$. Then $I + R_0(|\eta|^2 + i0)V$ is invertible on $L^{\infty}$ if and only if either $\eta \ne 0$ or $\eta = 0$ and zero is neither an eigenvalue, nor a resonance of $-\Delta + V$.

In this case, $I + R_0(|\eta|^2 \pm i\epsilon) V$ is also invertible on $L^{\infty}$ and its inverse is uniformly bounded in $\B(L^{\infty}, L^{\infty})$ for $\Im \eta > 0$.
\end{lemma}
By (\ref{2.29}), $I + R_0 V$ being invertible on $L^{\infty}$ is equivalent to $R_V$ being bounded from $L^{3/2, 1}$ to $L^{\infty}$.
\begin{proof}
Firstly, we note that $R_0(|\eta|^2 + i0) V$ is a bounded operator from $L^{\infty}$ to $C_0$ (the space of continuous functions that vanish at infinity). By the Arzel\`{a}-Ascoli Theorem, a set $A$ is precompact in $C_0$ if and only if:
\begin{list}{\labelitemi}{\leftmargin=1em}
\item[1.] it is equicontinuous:
\be\lb{1st_cond}
\forall \epsilon>0\ \exists \delta>0\ \forall |y|<\delta\ \forall a \in A\ \|a(\cdot -y) - a\|_{\infty} < \epsilon;
\ee
\item[2.] it has uniform decay at infinity:
\be\lb{2nd_cond}
\forall \epsilon>0\ \exists R\ \forall a \in A\ \|\chi_{|x|>R}(x) a(x)\|_{\infty} < \epsilon.
\ee
\end{list}

Then, assume that $V$ is smooth and compactly supported; consider the image through $R_0(|\eta|^2 + i0) V$ of a bounded set in $L^{\infty}$. (\ref{2nd_cond}) holds due to the compactness of $V$'s support. (\ref{1st_cond}) is implied by the gain in regularity due to convolution. By approximating $V \in L^{3/2, 1}$ with smooth, compactly supported potentials, compactness follows in this general case as well.



By Fredholm's alternative, $I + R_0(|\eta|^2 + i0) V$ is invertible if and only if the equation
$$
f = -V R_0(|\eta|^2 - i0) f
$$
has no nonzero solution within $L^1$. Assume it did; then let $g = R_0(|\eta|^2 - i0) f$; $g$ is in $L^{3, \infty}$ and satisfies the equation
$$
g= -R_0(|\eta|^2 - i0) V g.
$$
Since $V$ is real-valued, the results of Goldberg--Schlag \cite{golsch} imply that $g \in \langle x \rangle^{1/2-\epsilon} L^2$. By Ionescu--Jerison \cite{ionjer}, when $\eta \ne 0$ this leads to $g=0$.

When $\eta=0$, we refer to the spectral condition (\ref{1.10}) and to Lemma \ref{lem2.1}. Thus, the inverse $(I + R_0(|\eta|^2 + i0) V)^{-1}$ exists for every $\eta \in \set R^3$.

Due to the continuity of the mapping $\lambda \mapsto R_0(\lambda + i0) V \in \B(L^{\infty}, L^{\infty})$, the inverses have uniformly bounded norms when $\lambda$ is in a compact set.

By approximating $V$ with Schwartz-class potentials, we also obtain that $R_0(\lambda + i0) V \to 0$ as $\lambda \to \infty$, so $(I + R_0(\lambda + i0) V)^{-1}$ is bounded for $|\lambda|>>0$.

This extends to any set in the complex plane at a positive distance away from the eigenvalues --- in particular to the whole right half-plane.
\end{proof}

Next, we establish a framework in which $\oast$ is a bounded operation, which we use to express the relation between $T_{1+}^\epsilon$ and $T_+^\epsilon$.
\begin{lemma}\lb{lm2.4} Let
$$\begin{aligned}
Z = \{T(x_0, x_1, y) \mid & T \in \widehat {L^{\infty}_y} \B(L^{\infty}_{x_0}, L^{\infty}_{x_1})\}.
\end{aligned}$$
Then $Z$ is a Banach algebra under $\oast$.

If $V \in L^{3/2, 1}$ then $T_{1+}^\epsilon$ defined by (\ref{2.16}) is in $Z$ and 
$\mc F_y T_{1+}^\epsilon$ is given by
\be\lb{2.90}\begin{aligned}
\mc F_y T_{1+}^\epsilon(x_0, x_1, \eta) &= V(x_0) \frac {e^{i|x_0-x_1| \sqrt {|\eta|^2 + i\epsilon} - i(x_0-x_1) \eta}}{|x_0-x_1|} \\
&= V(x_0) e^{-ix_0 \eta} R_0(|\eta|^2+i\epsilon)(x_0, x_1) e^{ix_1 \eta}.
\end{aligned}\ee
If, in addition, $H = -\Delta+V$ has neither eigenvectors, nor resonances at zero, then $T_+^\epsilon$ 
also belongs to $Z$ and
\be\lb{2.37}
(I + T_{1+}^\epsilon) \oast (I - T_+^\epsilon) = (I - T_+^\epsilon) \oast (I + T_{1+}^\epsilon) = I.
\ee
\end{lemma}
The main shortcoming of using the space $Z$ is that $T_+$'s being in $Z$ does not imply the boundedness of the wave operators, so the proof of the latter fact uses a different space.
\begin{proof}[Proof of Lemma \ref{lm2.4}]
The algebra properties of $Z$ under $\oast$ follow from those of the ordinary composition of operators.

The Fourier transform of $T_{1+}^{\epsilon}$ is given by (\ref{2.90}) as a consequence of (\ref{2.16}) --- which contains the resolvent in $\xi_1$, translated in frequency ---, so the uniform boundedness follows from the pairing of $L^{3/2, 1}$ and $L^{3, \infty}$ functions.

$T_\epsilon^+$ is in $Z$ if its Fourier transform (\ref{eqn2.13}) is uniformly bounded, which is guaranteed by Lemma \ref{lemma2.3} if $V \in L^{3/2, 1}$ and $H$ fulfills the spectral condition~(\ref{1.10}).

Formula (\ref{2.37}) follows from the resolvent identity (\ref{2.29}), in which both sides are bounded on the right half-plane, away from eigenvalues.
\end{proof}

\subsection{Spaces of functions}
In this section we exhibit a structure for $W_{n+}$ that enables us to apply Wiener's theorem.


Recall that $s$ are elements of $\ISO(3) = \{s \in \B(\R^3, \R^3) \mid s^* s = I\}$, i.e. isometries: the identity operator, rotations, reflections, and improper rotations.

 
\begin{definition} For $g \in L^{\infty}$, $s \in \ISO(3)$, and $y \in \R^3$, let elementary transformations be mappings $e_{g, y, s}:L^p_x \to L^p_x$ of the form
$$
e_{g, y, s} f(x) := g(x) f(s x - y).
$$
\end{definition}
Note that elementary transformations are bounded on $L^p$, $1 \leq p \leq \infty$.

Our goal is to prove that $W_\pm$ and $W_\pm^*$ belong to the operator space $X$:
\begin{definition} Let $X$ the space of two-variable kernels
\be\begin{aligned}\lb{2.9}
X := &\{\frak X \in \B(L^{\infty}, L^{\infty})\mid (\frak X f)(x) = \int_{\R^3} \frak X(x, y) f(x-y)  \dd y, \\
&\frak X(x, y) = \int_{\ISO(3)} g_{s, y + x - s x}(x) \dd s, \\
&\int_{\R^3} \int_{\ISO(3)} \dd \|g_{s, y}\|_{L^{\infty}_x} < \infty\},
\end{aligned}\ee
\end{definition}

$g_{\omega, y}(x) \in \mc M_{\omega, y}$ is an $L^{\infty}$-valued measure, possibly singular in $y$ and $\omega$, such that $\|g_{\omega, y}\|_{L^{\infty}} \in \mc M_{\omega, y}$ is a finite-mass positive Borel measure.

$X$ is the space of integrable combinations of elementary transformations $e_{g, y, s}$. Each $\frak X \in X$ has the kernel $\frak X(x-y, x)$, i.e.
\be\lb{xf}\begin{aligned}
(\frak X f)(x) &= \int_{\R^3} \frak X(x, y) f(x-y)  \dd y \\
&= \int_{\ISO(3)} \int_{\set R^3} f(s x - y) \dd g_{s, y}(x).
\end{aligned}\ee
When $f$ is bounded (\ref{xf}) is absolutely convergent for each $x$.

Since elementary transformations are bounded on $L^p$ for $1 \leq p \leq \infty$, by Minkowski's inequality the same is true for every $\frak X \in X$. $\frak X \in X$ are distributions on $\R^6$ and $\|\frak X\|_{\B(L^p_{x_0}, L^p_{x_1})} \les \|\frak X\|_X$.

$X$ is also an algebra under operator composition --- a noncommutative semigroup algebra. Furthermore, we prove in the sequel that $W_{n+}^\epsilon$, $W_{\pm}$, and $W_{\pm}^*$ belong to $X$ for all $n$.

Even so, the composition operation in $X$ does not generate~$W_{n+}^\epsilon$. We introduce the extra structure of the space $Y$ defined below in order to recursively generate $W_{n+}^\epsilon$ and apply Wiener's theorem in $Y$.

\begin{definition} Let $Y$ be the space of three-variable kernels
\be\lb{2.31}\begin{aligned}
Y &:= \{T(x_0, x_1, y) \in Z \mid \forall s \in \ISO(3)\ T(x_0, x_1, y + x_0-s x_0) \in Z, \\
&\forall g \in L^{\infty}\ \forall s \in \ISO(3)\ \int_{\R^3} g(x_0) T(x_0, x, y+x_0-s x_0) \dd x_0 \in X\}.
\end{aligned}\ee
\end{definition}

We allow both singular and continuous locally integrable Borel measures to belong to $X$ or to $Y$, according to the definitions. Then both $X$ and $Y$ already contain their own identity elements, in~the~form~of
$$
I(x, y) = \delta_I(s) \delta_0(y)
$$
(times the constant function one in the $x$ variable) for $X$ and
\be\lb{ident}
I(x_0, x_1, y) = \delta_{x_0}(x_1) \delta_0(y) = \delta_{x_1}(x_0) \delta_0(y)
\ee
for $Y$.

For $\chi \in L^1_y$, we distinguish elements of the form $\frak X(x, y) = \chi(y) \in X$ given by $g_{y, s}(x) = \delta_I(s) \chi(y)$, respectively $T(x_0, x_1, y) = \delta_{x_0}(x_1) \chi(y) \in Y$. The latter represents convolution with $\chi$ within $Y$.

Except for the identity (\ref{ident}), elements of $Y$ that appear in the sequel have the further regularity property that
$$
\int_{\ISO(3)} \dd\|g_{s, y}(x)\|_{L^{\infty}_x} \in L^1_y.
$$
We can also add this property to the definition of $Y$ and put the identity element back in separately, as in \cite{becgol}.

The norms on $X$ and $Y$ are the natural ones:
$$\begin{aligned}
\|\frak X\|_X &= \inf\Big\{ \int_{\ISO(3)} \int_{\set R^3} d\|g_{s, y}\|_{L^{\infty}_x} \mid  \frak X(x, y) = \int_{\ISO(3)} g_{s, y + x - s x}(x) \dd \omega\Big\}
\end{aligned}$$
and
$$\begin{aligned}
\|T\|_Y &= \sup_{\substack{\|f\|_{\infty} = 1\\ s \in \ISO(3)}} \Big\|\int_{\R^3} f(x_0) T(x_0, x, y+x_0-s x_0) \dd x_0\Big\|_X.
\end{aligned}$$
Note that both norms are invariant under translation in $y$.

We prove the Wiener-type theorem in $Y$ because $Y$ has the required Wiener algebra structure; see Lemma \ref{yalg}.

For an elementary transformation $e_{g, y_0, s}$ and for $T(x_0, x_1, y) \in Y$, let the \emph{contraction} of $T$ by $e_{g, y_0, s}$ be
$$
(e_{g, y_0, s}T)(x, y) := \int_{\R^3} g(x_0) T(x_0, x, y + x_0 - s x_0) \dd x_0.
$$
The translation parameter $y_0$ need not enter the expression of $e_{g, y_0, s} T$.

Thus $Y$ consists of those three-variable kernels $T$ whose contraction $e_{g, y_0, s} T$ belongs to $X$ for any elementary transformation $e_{g, y_0, s}$. Note that
$$
\|e_{g, y_0, s} T\|_X \les \|g\|_{L^{\infty}_x} \|T\|_Y.
$$

The wave operators, which are in $X$, will be obtained by contracting elements of $Y$ as per Lemma \ref{lemma2.1}.

As a stronger alternative to integrability in (\ref{2.9}), we also use the weighted integrability condition
\be\lb{2.13}
\int_{\R^3} \langle y \rangle^{\beta} \Big(\int_{\ISO(3)} \dd \|g_{s, y}(x)\|_{L^{\infty}_x}\Big) \dd y < \infty.
\ee
Denote the spaces formed under condition (\ref{2.13}) for $0 \leq \beta < 1$ by $X_{\beta} \subset X$ and $Y_{\beta} \subset Y$:
\begin{definition} For $0 \leq \beta < 1$ let
$$\begin{aligned}
X_\beta := &\big\{ \frak X \in X \mid \frak X f(x) = \int_{\R^3} \frak X(x, y) f(x-y)  \dd y, \\
&\frak X(x, y) = \int_{\ISO(3)} g_{s, y + x - s x}(x) \dd \omega,\ (\ref{2.13}) \text{ holds}\big\},
\end{aligned}$$
respectively
$$\begin{aligned}
Y_\beta &:= \{T(x_0, x_1, y) \in Y \mid \forall g \in L^{\infty},\ \forall s \in \ISO(3),\\
&\int_{\R^3} g(x_0) T(x_0, x, y+x_0-s x_0) \dd x_0 \in X_\beta\}.
\end{aligned}$$
\end{definition}

We consider the range of spaces $X_\beta$ and $Y_\beta$ for $0\leq\beta<1$. $\beta \geq 0$ ensures that $X_{\beta}$ and $Y_{\beta}$ are algebras, while $\beta$ near to $1$ is almost optimal generically for the wave operator. See Lemma \ref{lema2.9} for results within weighted spaces.

For a rapidly decaying potential, the first term in the asymptotic expansion of $T_{1+}$ will be of order $\int_{\ISO(3)} \dd \|g_{s, y}\| \sim \langle y \rangle^{-4}$ by (\ref{2.81tertt}) and (\ref{2.81tert}), so it will not belong to~$Y_1$.

To reflect this, we define $X_{1+\epsilon}$ and $Y_{1+\epsilon}$ using the condition
\be\lb{cond_asimp}\begin{aligned}
&\exists g^1_{s, y}\ s.t.\ \int_{\ISO(3)} \dd\|g^1_{s, y}(x)\|_{L^{\infty}_x} \les \langle y \rangle^{-4},\\
&\int_{\R^3} \langle y \rangle^{1+\epsilon} \Big(\int_{\ISO(3)} \dd \|g_{s, y}(x) - g^1_{s, y}(x)\|_{L^{\infty}_x}\Big) \dd y < \infty.
\end{aligned}\ee
\begin{definition} For $0 < \epsilon < 1$ let
$$\begin{aligned}
X_{1+\epsilon} := &\big\{ \frak X \in X \mid \frak X f(x) = \int_{\R^3} \frak X(x, y) f(x-y)  \dd y, \\
&\frak X(x, y) = \int_{\ISO(3)} g_{s, y + x - s x}(x) \dd \omega,\ (\ref{cond_asimp}) \text{ holds}\big\},
\end{aligned}$$
respectively
$$\begin{aligned}
Y_{1+\epsilon} &:= \{T(x_0, x_1, y) \in Y \mid \forall g \in L^{\infty},\ \forall \omega \in S^2,\\
&\int_{\R^3} g(x_0) T(x_0, x, y) \dd x_0 \in X_{1+\epsilon},\\
&\int_{\R^3} g(x_0) T(x_0, x, y+x_0-s x_0) \dd x_0 \in X_{1+\epsilon}\}.
\end{aligned}$$
\end{definition}
By this definition, we isolate the first term in the asymptotic expansion of $T \in Y_{1+\epsilon}$ and state that all other terms have faster decay on average. We prove results in this setting by Lemma \ref{asimptotic}.

\begin{lemma}\lb{yalg}
$Y$ defined by (\ref{2.31}), $Y_{\beta}$ defined by (\ref{2.13}), and $Y_{1+\epsilon}$ defined by (\ref{cond_asimp}) are Banach algebras with the operation $\oast$, see~(\ref{eq2.21}). Moreover, $Y_{1+\epsilon} \subset Y_{\beta} \subset Y$.
\end{lemma}
Thus, provided $I + T_{1+}^\epsilon$ is invertible in $Y$, hence in $Z$, its inverse will be $I - T_+^\epsilon$ both in $Z$ and in $Y$, hence $T_+^\epsilon \in Y$.
\begin{proof}

The fact that $\oast$ is associative and non-commutative is clear in $Z$; the unit element is given by (\ref{ident}). Since $Y \subset Z$, the same is true in $Y$.

Performing a Fourier transform, $\oast$ takes the from
\be\lb{comp}
(\mc F_{x_0, x_2, y} T_1 \oast T_2)(\xi_0, \xi_2, \eta) = \int_{\set R^3} (\mc F_{x_0, x_1, y} T_1)(\xi_0, \xi_1, \eta) (\mc F_{x_1, x_2, y} T_2)(\xi_1, \xi_2, \eta) \dd \xi_1.
\ee
That is, the Fourier transform turns convolution in $y$ into pointwise multiplication in $\eta$, while the Fourier transform in $\xi_0$, $\xi_1$, and $\xi_2$ preserves the composition of operators in those variables.

Since kernels $\frak X \in X$ are integrable combinations of elementary transformations, for $\frak X \in X$ and $T \in Y$ we define the contraction $\frak X T$:
$$\begin{aligned}
(\frak X T)(x, y) &:= \int \frak X(x_0, y_0) T(x_0, x, y-y_0) \dd y_0 \dd x_0 \\
&=\int_{\R^6 \times \ISO(3)} T(x_0, x, y - y_0) \dd g_{s, y_0+x_0-s x_0}(x_0) \dd x_0 \\
&=\int_{\R^6 \times \ISO(3)} T(x_0, x, y - y_0 + x_0 - s x_0) \dd g_{s, y_0}(x_0) \dd x_0.
\end{aligned}$$

The definitions of $X$ and $Y$ imply that $\frak X T \in X$ and
\be\lb{xt}
\|\frak X T\|_X \les \|\frak X\|_X \|T\|_Y.
\ee

Take $T_1$ and $T_2$ in $Y$ and let $T_3 = T_1 \oast T_2$. Testing the definition of $Y$ for $T_3$, we see that
$$\begin{aligned}
&\int f(x_0) T_3(x_0, x_2, y + x_0 - s x_0) \dd x_0 = \\
&= \int_{\set R^9} f(x_0) T_1(x_0, x_1, y_1 + x_0 - s x_0) T_2(x_1, x_2, y-y_1) \dd x_1 \dd y_1 \dd x_0.
\end{aligned}$$
Integrating in $x_0$, in each case we obtain an expression of the form $\frak X T_2$ for $\frak X \in X$ with $\|\frak X\|_X \les \|T_1\|_Y$. This expression belongs to $X$ by (\ref{xt}). Thus, $T_3 = T_1 \oast T_2 \in Y$ and
$$
\|T_1 \oast T_2\|_Y \les \|T_1\|_Y \|T_2\|_Y.
$$

The same reasoning applies to $Y_{\beta}$. Note that $\langle y \rangle^{-\beta} L^1_y$ is an algebra under convolution for $\beta \geq 0$, by the triangle inequality. Then $X_\beta$ and consequently $Y_\beta$ are also algebras.

Finally, observe that
$$\begin{aligned}
\int_{\R^3} \langle x -y \rangle^{-4} \langle y \rangle^{-4} \dd y &= 2\int_{|x-y|\leq|y|} \langle x -y \rangle^{-4} \langle y \rangle^{-4} \dd y \\
&\les \int_{|y|\geq |x|/2} \langle x-y \rangle^{-4} \langle y \rangle^{-4} \dd y \\
&\les \langle x \rangle^{-4} \int_{\R^3} \langle x-y \rangle^{-4} \dd y \les \langle x \rangle^{-4}.
\end{aligned}$$
Likewise, if $f \in \langle x \rangle^{-1-\epsilon} L^1_x$,
$$
\int_{\R^3} \langle y \rangle^{-4} f(x-y) \dd y \leq \int_{|x-y|\leq|y|} \langle y \rangle^{-4} f(x-y) \dd y + \int_{|x-y|\geq|y|} \langle y \rangle^{-4} f(x-y) \dd y;
$$
then,
$$
\langle x \rangle^{1+\epsilon} \int_{|x-y|\geq|y|} \langle y \rangle^{-4} f(x-y) \dd y \les \int_{|x-y|\geq|y|} \langle y \rangle^{-4} \langle x-y \rangle^{1+\epsilon} f(x-y) \dd y < \infty
$$
and
$$
\int_{|x-y|\leq|y|} \langle y \rangle^{-4} f(x-y) \dd y \les \langle x \rangle^{-4} \int_{\R^3} f(x-y) \dd y \les \langle x \rangle^{-4}.
$$

This implies that $X_{1+\epsilon}$ and $Y_{1+\epsilon}$ are also algebras, under composition of operators and $\oast$ respectively.
\end{proof}

We use the following fact in defining the Fourier transform on $Y$:
\begin{lemma}\lb{lema2.4} For any $h \in L^{\infty}_y$, $\int T(x_0, x_1, y) h(y) \dd y$ defines a bounded operator in $\B(L^{\infty}_{x_0}, L^{\infty}_{x_1})$ and in $\B(L^1_{x_1}, L^1_{x_0})$ by
\be\lb{2.33}
\Big|\int f(x_0) T(x_0, x_1, y) h(y) g(x_1) \dd x_0 \dd y \dd x_1\Big| \les \|f\|_{\infty} \|T\|_Y \|h\|_{\infty} \|g\|_1.
\ee
\end{lemma}
\begin{proof}[Proof of Lemma \ref{lema2.4}] The conclusion will be true because
$$
\esssup_{x_1} \bigg| \int_{\R^3} \Big|\int_{\R^3} f(x_0) T(x_0, x_1, y) h(y) \dd x_0\Big| \dd y\bigg| \les \|f\|_{\infty} \|T\|_Y \|h\|_{\infty}.
$$
Performing the integral in $x_0$, for
$$
\frak X = \int_{\R^3} f(x_0) T(x_0, x_1, y) \dd x_0 \in X,\ \|\frak X\|_X \les \|f\|_{\infty} \|T\|_Y,
$$
we arrive at
$$
\esssup_{x_1} \Big|\int_{\R^3} \frak X(x_1, y) h(y) \dd y\Big| \les \|\frak X\|_X \|h\|_{\infty}.
$$
This is true because operators in $X$ are bounded on $L^{\infty}$.
\end{proof}

Next, in preparation for Wiener's theorem, we characterize the Fourier transform of elements $T \in Y$.
\begin{lemma}[Fourier transform]\lb{lemma2.14} For $T \in Y$, let
\be
\mc F_y T(x_0, x_1, \eta) = \int T(x_0, x_1, y) e^{-iy\eta} \dd \eta.
\ee
Then, for each $\eta$, $\mc F_y T(x_0, x_1, \eta)$ is a bounded $L^1$ operator:
\be\lb{2.89}
\esssup_{\eta} \|\mc F_y T(x_0, x_1, \eta)\|_ {\B(L^1_{x_1}, L^1_{x_0})} \les \|T\|_Y.
\ee
Conversely, consider a kernel $T$ having the product form
$$
T(x_0, x_1, y) := T_0(x_0, x_1) \chi(y),
$$
where $\chi \in L^1$ and $T_0 \in \B(L^\infty_{x_0}, L^\infty_{x_1})$. Then $T \in Y$ and
\be\lb{converse}
\|T\|_Y \les \|T_0\|_{\B(L^\infty_{x_0}, L^\infty_{x_1})} \|\chi\|_{L^1}.
\ee
\end{lemma}

\begin{proof}
Lemma \ref{lema2.4} directly implies that the Fourier transform $\mc F_y T(x_0, x_1, \eta)$ of an operator $T \in Y$ is $L^1$-bounded from $L^1_{x_1}$ to $L^1_{x_0}$.

Conversely, consider a kernel $T$ having the product form
\be
T = T_0(x_0, x_1) \chi(y).
\ee
Assume that $\chi \in L^\infty$ has a support of finite Lebesgue measure. By the definition (\ref{2.31}), $T$ is in $Y$ when for all $f \in L^{\infty}$ and $s \in \ISO(3)$ the following contracted operators are in $X$:
\be\lb{exp}
\int f(x_0) T(x_0, x_1) \chi(y + x_0 - s x_0) \dd x_0 \in X.
\ee
Then for each fixed $y$
$$\begin{aligned}
\|g(x_1, y)\|_{L^\infty_{x_1}} &:= \Big\|\int f(x_0) T(x_0, x_1) \chi(y+x_0-s x_0) \dd x_0\Big\|_{L^\infty_{x_1}} \\
&\les \|f\|_{L^\infty} \|T\|_{\B(L^\infty_{x_0}, L^\infty_{x_1})} \|\chi\|_{L^\infty}.
\end{aligned}$$
Integrating over the support of $\chi$ we obtain
$$
\|g(x_1, y)\|_X \les \int \|g(x_1, y)\|_{L^\infty_{x_1}} \dd y \les \|f\|_{L^\infty} \|T\|_{\B(L^\infty_{x_0}, L^\infty_{x_1})} \|\chi\|_{L^\infty} |\supp(\chi)|.
$$
Next, consider any $L^1$ function $\chi$ and decompose it as $\sum_k \chi_{2^k \leq |f(x)| < 2^{k+1}}(x) \chi(x)$. The conclusion follows.
\end{proof}

The following simple property of the space $Y$ is also useful in the sequel:
\begin{lemma}
$Y$ is translation-invariant in the $y$ variable: for any $y_0 \in \set R^3$
\be
\|T(x_0, x_1, y)\|_Y = \|T(x_0, x_1, y + y_0)\|_Y
\ee
and for any $\chi \in L^1$
$$
\|\chi *_y T\|_Y = \Big\|\int \chi(y_1) T(x_0, x_1, y-y_1) \dd y_1\Big\|_Y \leq \|\chi\|_1 \cdot \|T\|_Y.
$$
\end{lemma}
\begin{proof}
This property follows directly from the definition of $Y$ (\ref{2.31}).
\end{proof}

\subsection{The structure of the wave operators} In this section we set $\epsilon=0$ and show that the kernels $T_{n+}$, used to represent the operators $W_{n+}$ in the space $Z$ by formula (\ref{2.19}), have a specific structure captured by the definitions of $X$ and $Y$. In particular, then, $W_{n+} \in X$ for each $n$.

The structure of $W_{n+}$ is described in the next lemma.
\begin{lemma}\lb{lema2.13} If $V \in B$, then $T_{n+} \in Y$ and
\be\lb{2.83}
\|T_{n+}\|_Y \les C_2^n \|V\|_B^n.
\ee
If, in addition, $V \in \langle x \rangle^{-\alpha} L^2$, $1/2<\alpha<5/2$, and $0 \leq \beta < \alpha-1/2$, $\beta \ne 1$, then
\be\lb{eq2.84}
\|T_{n+}\|_{Y_{\beta}} \les C_2^n \|V\|_{\langle x \rangle^{-\alpha} L^2}^n.
\ee
\end{lemma}
Similar results hold for $W_{n+}^\epsilon$ when $\epsilon>0$, but the proof is more involved.
\begin{proof}
Firstly, we recall from Lemma \ref{lemma2.1} a specific formula for $W_{n+}^\epsilon$:
\be\begin{aligned}
\langle W_{n+}^\epsilon f, g \rangle &= (-1)^{n-1} \int_{\set R^6} T_{n+}^\epsilon(x_0, x, y) f(x-y) \ov g(x) \dd y \dd x \dd x_0 \\
&= (-1)^{n-1} \int_{\set R^6} (\mc F_{x_0, x, y} T_{n+}^\epsilon)(0, \xi, \eta) \widehat f(\xi) \ov {\widehat g}(\xi+\eta) \dd \eta \dd \xi,
\end{aligned}\ee
where the necessary and sufficient condition for $L^1$ boundedness is
\be\lb{2.36}
\sup_x \int \Big|\int T_{n+}^\epsilon(x_0, x, y) \dd x_0\Big| \dd y < \infty
\ee
and the kernels have the form (\ref{2.16}--\ref{2.17}).

By (\ref{eqn2.6}), the kernel associated to $W_{1+}$ is $T_{1+}$ given by
$$
(\mc F_{x_0, x, y} T_{1+})(\xi_0, \xi, \eta) = \lim_{\epsilon \downarrow 0} (\mc F_{x_0, x, y} T_{W_{1+}^\epsilon})(\xi_0, \xi, \eta) = \lim_{\epsilon \downarrow 0} \frac {\widehat V(\xi-\xi_0)}{|\xi+ \eta|^2 - |\eta|^2 - i \epsilon}.
$$
For $n \geq 1$, more generally (see (\ref{eqn2.7})), $T_{n+}^\epsilon$ represents $W_{n+}^\epsilon$:
\be\begin{aligned}\lb{2.56}
(\mc F_{x_0, x_n, y} T_{n+})(\xi_0, \xi_n, \eta) &= \lim_{\epsilon \downarrow 0} (\mc F_{x_0, x_n, y} T_{n+}^\epsilon)(\xi_0, \xi_n, \eta) \\
&= \lim_{\epsilon \downarrow 0} \int \frac{\prod_{\ell=1}^n \widehat V(\xi_{\ell} - \xi_{\ell-1}) \dd \xi_1 \ldots \dd \xi_{n-1}}{\prod_{\ell=1}^n (|\xi_{\ell}+\eta|^2-|\eta|^2 - i \epsilon)}.
\end{aligned}\ee

The subsequent computations can be carried on in full generality for $n \geq 1$, but become quite involved when $n > 1$, see \cite{yajima0}. Thus, we spare a considerable effort by directly proving (\ref{2.83}) and (\ref{eq2.84}) only when $n=1$ and then using the algebra structure of $Y$ and $Y_{\beta}$ to infer that the same results hold for $n>1$.

Reversing the Fourier transform in (\ref{2.56}), note that the expression will contain a convolution in the $y$ variable (resulting from a product in $\eta$). Effectively, for $\epsilon>0$ and any Schwartz function $f$
\be\lb{eqn2.57}
\Big(\mc F_{\eta} \int_{\R^3} \frac {f(\xi)}{|\xi+\eta|^2-|\eta|^2-i\epsilon}\dd\xi \Big) (t \omega) = \int_0^{\infty} f(s \omega) e^{-its/2-\epsilon t/(2s)} s \dd s.
\ee
At this point we assume that $V$ is of Schwartz class; then, we can let $\epsilon$ become zero in (\ref{eqn2.57}), thus in (\ref{2.56}). We find that $W_{1+}$ has the form
\be\begin{aligned}\lb{eqn2.90}
W_{1+} f(x) = \int_{S^2} \int_{[0, \infty)} &K_1(x, t \omega) f(x+t \omega) \dd t \dd \omega,
\end{aligned}\ee
where $K_1$ can be written in polar coordinates as
\be\begin{aligned}\lb{2.50}
K_1(x, t\omega) = \frac 12\int_{[0, \infty)} \widehat {V}(s \omega) e^{-it s/2} e^{i s \omega \cdot x} s \dd s.
\end{aligned}\ee
We transform the $x$ dependence in (\ref{2.50}) by letting
\be\begin{aligned}\lb{2.60}
L_1(t\omega) := \int_{[0, \infty)} \widehat {V}(s \omega) e^{-it s/2} s \dd s,
\end{aligned}\ee
respectively
\be\begin{aligned}\lb{2.61}
\tilde L_1(t\omega) := \int_{(-\infty, 0]} \widehat {V}(s \omega) e^{-it s/2} s \dd s.
\end{aligned}\ee
$L_1$ and $\tilde L_1$ involve the same integrand, but integrated over different regions.

With this notation, we distinguish between two situations in (\ref{2.50}), namely $x \cdot \omega \leq t/2$ and $x \cdot \omega \geq t/2$. In the first situation, $(t-2x\cdot \omega) \omega$ has the same orientation as $\omega$. In the second situation, $(t-2x\cdot \omega) \omega$ and $\omega$ have opposite orientations.

Combining the two cases, one has that
$$\begin{aligned}
K_1(x, t \omega) &= \frac 12\chi_{(-\infty, \frac {t} 2)}(x \cdot \omega)\ L_1((t-2x\cdot \omega) \omega) +\\
&+ \frac 12\chi_{(\frac {t} 2, \infty)}(x \cdot \omega)\ \tilde L_1((t - 2x \cdot \omega)\omega).
\end{aligned}$$

Under suitable conditions on $V$, the next lemma, Lemma \ref{lema2.9}, shows that both $L_1$ and $\tilde L_1$ are integrable. For $\omega \in S^2$, let $S_\omega(x) = x - 2(x \cdot \omega) \omega \in \ISO(3)$. Then, for each $x$, let $g_{\omega, y}(x)$ be the measure on $\ISO(3) \times \set R^3$ supported on the codimension-three subset
$$
\supp g_{s, y}(x) = \{(S_\omega, t\omega) \mid \omega \in S^2,\ t \leq 0\}
$$
and given on this set by
$$\begin{aligned}
dg_{S_\omega, -t\omega}(x) &= \frac 12\big(\chi_{(-\infty, t/2)}(x \cdot \omega)\ L_1(t\omega) + \chi_{(t/2, \infty)}(x \cdot \omega)\ \tilde L_1(t\omega)\big) \dd t \dd \omega.
\end{aligned}$$
Othewise put,
\be\lb{2.100}
\dd g_{S_\omega, -y}(x) = \frac 12\delta_{\frac y {|y|}}(\omega) \big(\chi_{(-\infty, |y|/2)}(x \cdot \omega)\ L_1(y) + \chi_{(|y|/2, \infty)}(x \cdot \omega)\ \tilde L_1(y)\big) |y|^{-2} \dd y.
\ee
Finally, we can rewrite (\ref{eqn2.90}) as
\be\lb{2.32}
(W_{1+} f)(x) = \int_{\ISO(3)} \int_{\set R^3} f(s x - y) \dd g_{s, y}(x),
\ee
where by Lemma \ref{lema2.9}
\be\lb{eq2.33}\begin{aligned}
\int_{\set R^3} \Big(\int_{\ISO(3)} d \|g_{s, y}(x)\|_{L^{\infty}_x}\Big) \dd y &\leq \int \big(|L_1(y)| + |\tilde L_1(y)|\big) |y|^{-2} \dd y \les \|V\|_B < \infty.
\end{aligned}\ee
In particular, the $L^1$-boundedness of $T_{1+}$, (\ref{2.36}), is implied by the strictly stronger assertions (\ref{2.32}) and (\ref{eq2.33}).

To emphasize the dependence of $W_{1+}$ and $T_{1+}$ on the potential $V \in B$,~let
\be\lb{xfrak}\begin{aligned}
\frak X_{V+}(x, y) &:= W_{1+} = \int_{\R^3} T_{1+}(x_0, x, y) \dd x_0, \\
\frak X_{V-}(x, y) &:= W_{1-} = \int_{\R^3} T_{1-}(x_0, x, y) \dd x_0.
\end{aligned}\ee
Using the notation of Lemma \ref{lemma2.1}, we rephrase (\ref{2.32}) and (\ref{2.33}) as
\be\lb{eq2.98}
\|\frak X_{V+}(x, y)\|_X = \Big\|\int_{\R^3} T_{1+}(x_0, x, y) \dd x_0\Big\|_X \les \|V\|_B.
\ee

Based on (\ref{eq2.98}), we next show that $T_{1+} \in Y$, by checking that $eT_{1+} \in X$ for every elementary transformation $e$.

For any $f \in L^{\infty}$ let $V_f = f(-x) V(x)$ and note that, since $B$ is a Banach lattice, $\|V_f\|_B \les \|f\|_{\infty} \|V\|_B$.

We show that contractions $e_{f, y, s} T_{1+}$ have a well-determined form. To begin with the simpler case of $e_{f, y, I}$ (i.e.\ $s=I \in \ISO(3)$),
$$\begin{aligned}
e_{f, y_0, I} T_{1+} &= \int_{\R^3} f(x_0) T_{1+}(x_0, x, y) \dd x_0 \\
&= \mc F^{-1}_{x, y} \int \frac {\widehat f(-\xi_0) \widehat V(\xi-\xi_0)}{|\xi+\eta|^2-|\eta|^2-i0} \dd \xi_0 = \frak X_{V_f+} \in X.
\end{aligned}$$
Regarding $e_{f, y_0, s}$ in general, for $s \in \ISO(3)$ we evaluate
$$
E(x, y) := \int f(x_0) T_{1+}(x_0, x, y + x_0 - s x_0) \dd x_0.
$$
For a matrix $A$ of determinant one, $\mc F \{f(Ax)\} = \widehat f((A^{-1})^t \xi)$. Thus
$$
A \bpm x \\ y \epm = \bpm x \\ y+x-sx \epm \implies
(A^{-1})^t \bpm \xi \\ \eta \epm = \bpm \xi - (\eta-s^{-1}\eta) \\ \eta \epm.
$$
Performing the Fourier transform
$$\begin{aligned}
E &= \int_{\R^3} f(x_0) T_{1+}(x_0, x, y + x_0 - sx_0) \dd x_0\\
&= \int_{\R^3} f(x_0) \mc F^{-1}_{x_0, x, y} \Big(\frac {\widehat{V}(\xi-\xi_0)}{|\xi+\eta|^2-|\eta|^2-i0}\Big)(x_0, x, y + x_0-sx_0) \dd x_0 \\
&= \int_{\R^3} f(x_0) \mc F^{-1}_{x, y} \Big(\frac {\widehat{V}(\xi-\xi_0+\eta-s^{-1}\eta)}{|\xi+\eta|^2-|\eta|^2-i0}\Big)(x_0, x, y) \dd x_0 \\
&= \mc F^{-1}_{x, y} \Big(\frac {\widehat{V_f}(\xi+\eta-s^{-1}\eta}{|\xi+\eta|^2-|\eta|^2-i0}\Big)(x, y).
\end{aligned}$$

Letting $\tilde \xi = \xi + \eta - s^{-1} \eta$, $\tilde \eta = s^{-1} \eta$,

$$\begin{aligned}
\mc F_{x, y} E&= \frac {\widehat {V_f}(\xi + \eta-s^{-1}\eta)}{|\xi + \eta|^2 - |\eta|^2 - i0} = \frac {\widehat {V_f}(\tilde \xi)}{|\tilde \xi + \tilde \eta|^2 - |\tilde \eta|^2 - i0} \\
&= (\mc F_{x, y} \frak X_{fV+})(\xi + \eta-s^{-1}\eta, s^{-1}\eta).
\end{aligned}$$
$\frak X_{fV+}$, corresponding to the potential $fV$, belongs to $X$ by Lemma \ref{lema2.9}.
Note that
$$
A \bpm \xi \\ \eta \epm = \bpm \xi + \eta-s^{-1}\eta \\ s^{-1}\eta \epm \implies (A^{-1})^t \bpm x \\ y \epm = \bpm x \\ s^{-1}y + x-s^{-1}x \epm.
$$
The expression becomes
$$\begin{aligned}
E &= \mc F^{-1}_{x, y} \big(\mc F_{x, y} \frak X_{V_f+}(\xi + \eta - s \eta, s \eta)\big) \\
&= \frak X_{V_f+}(x, s^{-1}y + x-s^{-1}x).
\end{aligned}$$
Let $\frak X_{V_f+}$ have the kernel
$$
\frak X_{V_f+}(x, y) = \int_{\ISO(3)} g_{\sigma, y+x-\sigma x}(x) \dd \sigma,
$$
i.e. (see(\ref{xf}))
$$
(\frak X_{V_f+} h)(x) = \int_{\ISO(3)} \int_{\R^3} h(\sigma x-y) \dd g_{\sigma, y}(x).
$$
Then $\frak X_{V_f+}(x, s^{-1}y + x-s^{-1}x)$ is going to have the kernel
$$
\frak X_{V_f+}(x, sy + x-sx) = \int_{\ISO(3)} g_{\sigma, s^{-1}(y+x-\sigma x)+x-s^{-1}x}(x) \dd \sigma,
$$
so that
$$
(\frak X_{V_f+}(x, s^{-1}y + x-s^{-1}x) h)(x) = \int_{\ISO(3)} \int_{\R^3} h(s^{-1}\sigma x-s^{-1}y) \dd g_{\sigma, y}(x).
$$
Thus the measure $g_{\sigma, y}$ is replaced by $g_{s \sigma, s y}$, which has exactly the same properties as the original --- in particular,
$$
\int_{\R^3} \int_{\ISO(3)} \dd \|g_{s \sigma, s y}\|_{L^\infty_x} = \int_{\R^3} \int_{\ISO(3)} \dd \|g_{\sigma, y}\|_{L^\infty_x}.
$$
This shows that $\|\frak X_{V_f+}(x, s^{-1}y + x-s^{-1}x)\|_X = \|\frak X_{V_f+}\|_X$. Thus, $E \in X$, with control of the norms at every step.

This implies that $T_{1+} \in Y$ and, recursively, that $T_{n+} \in Y$ for all $n \geq 1$.

Henceforth we no longer assume that $V$ is of Schwartz class. For general $V \in B$, consider a sequence of approximations $V_n \in \mc S$.

On one hand, $T_{1+}(V_n)$ form a Cauchy sequence in $Y$, as
$$
\|T_{1+}(V_n) - T_{1+}(V_m)\|_Y \les \|V_n-V_m\|_B,
$$
Since $Y$ is a complete metric space, $T_{1+}(V_n)$ then have a limit in $Y$.

On the other hand, $V_n \to V$ in $B$ implies that $V_n \to V$ in $L^{3/2, 1}$, so $T_{1+}(V_n) \to T_{1+}(V)$ in $Z$. But the limit must be the same in $Y$ and in $Z$ since $Y \subset Z$, so $T_{1+}(V) \in Y$ and
$$
\|T_{1+}(V)\|_Y = \lim_{n \to \infty} \|T_{1+}(V_n)\|_Y \les \|V\|_B.
$$

Analogously, for $V\in\langle x \rangle^{-\alpha} L^2$, $1/2<\alpha<3/2$, Lemma \ref{lema2.9} shows that $T_{1+} \in Y_{\beta}$, as per~(\ref{2.81}) and (\ref{2.81bis}). For $3/2<\alpha<5/2$, we obtain by (\ref{2.81tert}) and (\ref{2.81tertt}) that $T_{1+} \in Y_{1+\epsilon}$, where $\epsilon = \alpha-1/2$.
\end{proof}

Next, we prove that $L_1$ and $\tilde L_1$ are integrable, as required for (\ref{2.33}). A similar representation exists for $W_{n+}$ when $n > 1$, but the computations are more involved.
\begin{lemma}\lb{lema2.9} For $L_1$ and $\tilde L_1$ defined by (\ref{2.60}), respectively (\ref{2.61}),
\be\begin{aligned}
\int_{S^2} \int_{[0, \infty)} &|L_1(t \omega)| \dd t \dd \omega \les \|V\|_B
\end{aligned}\ee
Moreover, when $V \in \langle x \rangle^{-1} L^2$,
\be\begin{aligned}\lb{2.81}
\int_{S^2} \int_{[0, \infty)} &|L_1(t \omega)|^2  \langle t \rangle^2 \dd t \dd \omega \les \|V\|_{\langle x \rangle^{-1} L^2}^2,
\end{aligned}\ee
and if $V \in \langle x \rangle^{-\alpha} L^2$, $1/2 < \alpha < 3/2$,
\be\begin{aligned}\lb{2.81bis}
\int_{S^2} \int_{[0, \infty)} &|L_1(t \omega)|^2 \langle t \rangle^{2\alpha} \dd t \dd \omega \les \|V\|_{\langle x \rangle^{-\alpha} L^2}^2.
\end{aligned}\ee
Finally, if $V \in \langle x \rangle^{-3/2-\epsilon} L^2$, then there exists
\be\lb{2.81tertt}
L_1^1(t\omega) = \tilde L_1^1(t\omega) = 4 \langle t \rangle^{-2} \widehat V(0)
\ee
such that
\be\lb{2.81tert}
\begin{aligned}
\int_{S^2} \int_{[0, \infty)} &|L_1(t \omega)-L_1^1(t\omega)|^2 \langle t \rangle^{3+2\epsilon} \dd t \dd \omega \les \|V\|_{\langle x \rangle^{-3/2-\epsilon} L^2}^2
\end{aligned}
\ee
and same for $\tilde L_1$.
\end{lemma}
\begin{proof} We prove this following Yajima \cite{yajima0}, for $L_1$ only, where
$$
L_1(t \omega) = \int_0^{\infty} \widehat V(s\omega) e^{-its/2} s \dd s.
$$
By Plancherel's identity,
$$
\int |L_1(t\omega)|^2 \dd t \les \int |\widehat V(s\omega)|^2 s^2 \dd s.
$$
Integrating over $\omega \in S^2$,
\be\lb{2.84}
\|L_1\|_{L^2_{t, \omega}}^2 = \int |L_1(t\omega)|^2 \dd t \dd \omega \les \|\widehat V\|_2^2 = C \|V\|_2^2.
\ee
Integrating by parts, we likewise obtain, from
$$
\frac {it} 2 L_1(t \omega) = \int_0^{\infty} \partial_s (\widehat V(s\omega) s) e^{-its/2} \dd s,
$$
that
\be\lb{2.85}
\|t L_1\|_{L^2_{t, \omega}}^2 = \int |t L_1(t\omega)|^2 \dd t \dd \omega \les (\|\dl \widehat V\|_2^2 + \||\eta|^{-1} \widehat V\|_{L^2_{\eta}}^2) \les \|V\|_{|x|^{-1} L^2}^2.
\ee
When $V \in \langle x \rangle^{-1} L^2$, both (\ref{2.84}) and (\ref{2.85}) are valid, so we combine them and obtain (\ref{2.81}).

One more integration by parts yields that
$$
- \frac {t^2} 4 L_1(t \omega) = \int_0^{\infty} \partial_s^2 (\widehat V(s\omega) s) e^{-its/2} \dd s - \widehat V(0).
$$
If $\widehat V(0) = 0$, it follows that
\be\lb{2.105}
\|t^2 L_1\|_{L^2_{t, \omega}}^2 \les (\|\dl^2 \widehat V\|_2^2 + \||\eta|^{-1} \dl \widehat V\|_{L^2_{\eta}}^2).
\ee
However, the condition $\widehat V(0) = 0$ is unstable under perturbations, so the best possible rate of decay in (\ref{2.105}) --- for generic Schwartz potentials, for example --- is $t^{-2}$. Further integrating by parts, one obtains an asymptotic expansion:
$$
L_1(t \omega) = \Big(\frac {it} 2\Big)^{-n} \int_0^{\infty} \partial_s^n (\widehat V(s\omega) s) e^{-its/2} \dd s + \sum_{\ell=0}^{n-1} \Big(\frac {it} 2\Big)^{-\ell-1} \partial^{\ell}_s (\widehat V(s\omega) s) \mid_{s=0}.
$$

We next exhibit an almost-optimal decay rate of $t^{-3/2+\epsilon} L^2_t$ for generic potentials.
Indeed, following (\ref{2.85}),
\be\lb{2.85bis}
\|t^{3/2-\epsilon} L_1(t\omega)\|_{L^2_{t, \omega}}^2 \les \|\chi_{[0, \infty)}(s) \widehat V(s\omega) s\|_{L^2_{\omega} \dot H^{3/2-\epsilon}_s} \les_{\epsilon} \|V\|_{|x|^{-3/2+\epsilon} L^2}.
\ee
To show the second inequality, begin by assuming that $V \in \langle x \rangle^{-2} L^2$ and $\widehat V(0) = 0$. Then
$$
\|\partial^2_{|\eta|} \widehat V(\eta)\|_{L^2_{\eta}} \les \|V\|_{|x|^{-2} L^2} \text{ and } \|\widehat V(\eta)\|_{L^2_{\eta}} = C \|V\|_{L^2}.
$$
By interpolation, we get that for all $V \in \langle x \rangle^{-2} L^2$ such that $\widehat V(0) = 0$
$$
\|s \partial^{3/2-\epsilon}\big(\chi_{[0, \infty)}(s) \widehat V(s\omega)\big)\|_{L^2_{s, \omega}} \les \|V\|_{|x|^{-3/2+\epsilon} L^2}.
$$
The set $\{V \in \langle x \rangle^{-2} L^2 \mid \widehat V(0) = 0\}$ is dense in $|x|^{-3/2+\epsilon} L^2$, so the conclusion obtains for all $V$ in $|x|^{-3/2+\epsilon} L^2$. Finally, by Leibniz's rule,
$$\begin{aligned}
\|\partial^{3/2-\epsilon}\big(s \chi_{[0, \infty)}(s) \widehat V(s\omega)\big)\|_{L^2_{s, \omega}} &\les\big(\|s \partial^{3/2-\epsilon}\big(\chi_{[0, \infty)}(s) \widehat V(s\omega)\big)\|_{L^2_{s, \omega}} \\
&+ \|\partial^{1/2-\epsilon}\big(\chi_{[0, \infty)}(s) \widehat V(s\omega)\big)\|_{L^2_{s, \omega}}\big).
\end{aligned}$$
The second term can be bounded by similar means and by using Hardy's inequality. We retrieve the second part of (\ref{2.85bis}).

For $V \in \langle x \rangle^{-\alpha} L^2$, both (\ref{2.84}) and (\ref{2.85bis}) are valid, so we obtain (\ref{2.81bis}).

(\ref{2.81tert}) is proved in the same manner, after isolating the leading-order term at infinity in the asymptotic expansion.

On the other hand, when $V \in B$ we apply the real interpolation method, see \cite{bergh}. Begin by partitioning $L_1$ into dyadic pieces
$$
L_{1j}(t\omega) = L_1(t\omega) (\chi(2^{-j-1} t) - \chi(2^{-j+1} t)).
$$
When $V \in L^2$, we rewrite (\ref{2.84}) as
$$
L_{1j}(t\omega) \in \ell^2_j(L^2_{t, \omega}),\ \|L_{1j}(t\omega)\|_{\ell^2_j(L^2_{t, \omega})} \les \|V\|_2.
$$
Likewise, (\ref{2.85}) becomes
$$
L_{1j}(t\omega) \in 2^{-j} \ell^2_j(L^2_{t, \omega}),\ \|L_{1j}(t\omega)\|_{2^{-j} \ell^2_j(L^2_{t, \omega})} \les \|V\|_{|x|^{-1} L^2}.
$$
By real interpolation we obtain, since $B$ was chosen precisely to be the real interpolation space $B = (L^2, |x|^{-1} L^2)_{(\frac 1 2, 1)}$, that
$$
\|L_{1j}(t\omega)\|_{2^{-j/2} \ell^1_j(L^2_{t, \omega})} \les \|V\|_B,
$$
hence
$$
\|L_1\|_{L^1_{t, \omega}} \les \|L_{1j}(t\omega)\|_{\ell^1_j(L^1_{t, \omega})} \les\|L_{1j}(t\omega)\|_{2^{-j/2} \ell^1_j(L^2_{t, \omega})} \les \|V\|_B.
$$
\end{proof}

We identify the subspace of elements of $Y$ with two specific properties, which we call \emph{continuity} and \emph{decay at infinity}.

In particular, for every $V \in B$, $(\frak X_{V+})^2 := T_{1+} \oast T_{1+} = T_{2+}$ has them:
\begin{lemma}\lb{lm2.13} Assume $V \in B$ and let $T_{1+}$ be given by (\ref {2.16}). Then
\be\lb{2.73}
\lim_{y_0 \to 0} \|(T_{1+} \oast T_{1+})(x_0, x_1, y + y_0) - (T_{1+} \oast T_{1+})(x_0, x_1, y)\|_Y = 0
\ee
and, uniformly for all contractions $e_{f, 0, s}$ with $\|f\|_{L^\infty} \leq 1$, $s \in \ISO(3)$,
\be\lb{2.74}\begin{aligned}
&\lim_{R \to \infty} \sup_{\substack{\|f\|_{L^\infty} \leq 1\\ s \in \ISO(3)}} \|(1-\chi(y/R)) e_{f, 0, s} T_{1+}(x_0, x_1, y)\|_X = \\
&= \lim_{R \to \infty} \sup_{\substack{\|f\|_{L^\infty} \leq 1\\ s \in \ISO(3)}} \|(1-\chi(y/R)) \frak X_{V_f+}(x_1, s^{-1} y + x_1 - s^{-1} x_1)\|_X = 0.
\end{aligned}\ee
\end{lemma}
(\ref{2.73}) and (\ref{2.74}) together or separately define Banach subalgebras of~$Y$.

For $T \in Y$, let
$$
\mc F_y T(x_0, x_1, \eta) = \int_{\R^3} T(x_0, x_1, y) e^{-iy\eta} \dd y.
$$
(\ref{2.73}) and (\ref{2.74}) are useful because we can localize on the Fourier side when (\ref{2.74}) holds and we can disregard the tail at infinity on the Fourier side when (\ref{2.73}) holds.

We take both (\ref{2.73}) and (\ref{2.74}) as hypotheses in the abstract version of Wiener's theorem.

\begin{proof} 
Condition (\ref{2.74}) is clearly true when $V \in \langle x \rangle^{-1} L^2$, so $T_{1+} \in Y_1$ and all contractions of $T_{1+}$ have the same specified decay rate at infinity.

Thus, when $V \in B$, we approximate it by potentials in $\langle x \rangle^{-1} L^2$ and (\ref{2.74}) follows.

Translation by $y_0$ in (\ref{2.73}) exactly corresponds to translation by $y_0$ of the measure $g_{s, y}(x)$ associated to $W\in X$, see (\ref{2.9}) and (\ref{2.31}).

Then (\ref{2.73}) is implied by
\be\begin{aligned}\lb{cnd}
\lim_{y_0 \to 0} & \int_{\R^3} \int _{\ISO(3)} \esssup_x |g_{s, y}(x) - g_{s, y+y_0}(x)| \dd s \dd y = 0.
\end{aligned}\ee

For $T_{1+}$, $g_{s, y}(x)$ is a singular measure in $(s, y)$. For fixed $s=S_\omega$, the support of $g_{s, y}(x)$ has codimension two, being the half-line $y=-t\omega$, $t \geq 0$; see (\ref{2.100}).

Since the support is singular, but not punctual, repeated convolutions make $g_{s, y}(x)$ smoother, hence $T_{2+}$ is more regular.

For $W_{n+}$, $n >1$, the analogous measure has the formula
\be\begin{aligned}\lb{2.134}
g_{s=S_{\omega_n}, y}(x) = \int_{\substack{y = y_1 + \ldots + y_{n-1} + t_n \omega_n}} \chi_{(-\infty, t_n/2)}(x \cdot \omega_n)\ \frac{L_n(y_1, \ldots, t_n \omega_n)}{|y_1|^2 \ldots |y_{n-1}|^2} + \\
+ \chi_{(t_n/2, \infty)}(x \cdot \omega_n)\ \frac{\tilde L_n(y_1, \ldots, t_n \omega_n)}{|y_1|^2 \ldots |y_{n-1}|^2} \dd y_1 \ldots \dd y_{n-1} \dd t_n,
\end{aligned}\ee
where
$$\begin{aligned}
L_n(t_1\omega_1, \ldots, t_n \omega_n) = \int_{[0, \infty)^n} &\prod_{k=1}^n \widehat {V}(s_k \omega_k - s_{k-1} \omega_{k-1}) e^{-it_k s_k/2} \\
&s_1 \ldots s_n \dd s_1 \ldots \dd s_n,
\end{aligned}$$
respectively
$$\begin{aligned}
\tilde L_n(t_1\omega_1, \ldots, t_n \omega_n) = \int_{[0, \infty)^{n-1}} \int_{(-\infty, 0]} &\prod_{k=1}^n \widehat {V}(s_k \omega_k - s_{k-1} \omega_{k-1}) e^{-it_k s_k/2} \\
&s_1 \ldots s_n \dd s_1 \ldots \dd s_n.
\end{aligned}$$
Assume that $V \in \mc S$ and take $n=2$. Then
$$\begin{aligned}
g_{s=S_{\omega_2}, y}(x) = \int_{\substack{y = y_1 + t_2 \omega_2}} \chi_{(-\infty, t_2/2)}(x \cdot \omega_2)\ \frac{L_2(y_1, t_2 \omega_2)}{|y_1|^2} + \\
+ \chi_{(t_2/2, \infty)}(x \cdot \omega_2)\ \frac{\tilde L_2(y_1, t_2 \omega_2)}{|y_1|^2} \dd y_1 \dd t_2.
\end{aligned}$$
Condition (\ref{cnd}) reduces to
$$\begin{aligned}
\lim_{y_0 \to 0} \int_{(\R^3)^2} &\Big|\frac{L_2(y_1+y_0, y_2)}{|y_1+y_0|^2 |y_2|^2} - \frac{L_2(y_1, y_2)}{|y_1|^2 |y_2|^2}\Big| + \\
&+ \Big|\frac{\tilde L_2(y_1+y_0, y_2)}{|y_1+y_0|^2 |y_2|^2} - \frac{\tilde L_2(y_1, y_2)}{|y_1|^2 |y_2|^2}\Big| \dd y_1 \dd y_2 = 0.
\end{aligned}$$
With no loss of generality consider only $L_2$. As in Lemma \ref{lema2.9} or \cite{yajima0}, it is the case that
$$
\int_{(\R^3)^2} \Big|\frac{L_2(y_1, y_2)}{|y_1|^2 |y_2|^2}\Big| \les \|V\|_B^2.
$$
Then it suffices to show that for any $\epsilon>0$, $R< \infty$
$$
\int_{\{y_1 \mid \epsilon < |y_1| < R\} \times \R^3} \frac{|\partial_{y_1} L_2(y_1, y_2)|}{|y_1|^2 |y_2|^2} \dd y_1 \dd y_2 < \infty.
$$
Note that
$$\begin{aligned}
L_2(t_1\omega_1, t_2 \omega_2) = \int_{[0, \infty)^2} \widehat V(s_2 \omega_2 - s_1 \omega_1) \widehat V(s_1\omega_1) e^{-it_2 s_2/2-it_1s_1/2} s_1 s_2 \dd s_1 \dd s_2.
\end{aligned}$$
Differentiating this expression in $t_1$ or $\omega_1$ leads at most to an extra factor of $s_1$ and an extra derivative on $\widehat V(s_1 \omega_1)$, so for $\widehat V \in \mc S$
$$
\int_{\{y_1 \mid \epsilon < |y_1| < R\} \times \R^3} \frac{|\partial_{y_1} L_2(y_1, y_2)|}{|y_1|^2 |y_2|^2} \dd y_1 \dd y_2 < \infty.
$$
This suffices to prove (\ref{2.73}) for $V \in \mc S$. Again by approximation, we obtain that (\ref{2.73}) holds for any $V \in B$.
\end{proof}



\subsection{Proofs of the main statements}
We next prove the principal result, Theorem \ref{theorem_1.1}.
\begin{proof}[Proof of Theorem \ref{theorem_1.1}]
From the beginning, we set $\epsilon = 0$.



Let $\chi$ be a smooth cutoff function such that $\widehat \chi(\eta) = 1$ on $B(0, 1)$ and $\supp \widehat \chi \subset B(0, 2)$. Denote
$$
\chi_{\delta}(y) = \delta^{3} \chi(\delta y), \chi_R(y) = R^{3} \chi(R y).
$$
Then $\widehat{\chi_{\delta}}(\eta) = \widehat{\chi}(\delta^{-1} \eta)$ and $\widehat{\chi_R}(\eta) = \widehat{\chi}(R^{-1} \eta)$.

Firstly, due to (\ref{2.73}) from Lemma \ref{lm2.13}, we can cut the high frequencies off of $T_{1+} \oast T_{1+}$ by convolution with a localized function:
\be\lb{2.120}
\lim_{R \to \infty} \|\chi_R(y) * (T_{1+} \oast T_{1+}) - T_{1+} \oast T_{1+}\|_Y = 0.
\ee
We use (\ref{2.120}) to show that, for any sufficiently large $R$,
\be\lb{2.121}
(I-\chi_R(y)) * (I+T_{1+})^{-1} \in Y.
\ee
Indeed, consider two large radii $R_1$ and $R_2$ with $2R_1<R_2$. Then
$$\begin{aligned}
&(I-\chi_{R_2}(y)) * (I + T_{1+})^{-1} =\\
&= (I-\chi_{R_2}(y)) * \big(I - (I-\chi_{R_1}(y))*(T_{1+} \oast T_{1+})\big)^{-1} \oast (I - T_{1+}).
\end{aligned}$$
This formal computation is legitimate if conducted in the space $Z$, instead of $Y$. By (\ref{2.120}), for large enough $R_1$, the series
\be
(I + (I-\chi_{R_1}(y))*T_{1+}\oast T_{1+})^{-1} = \sum_{\ell=0}^{\infty} (-1)^{\ell} \big((I-\chi_{R_1}(y))*T_{1+}\oast T_{1+}\big)^{\ell}
\ee
converges in $Y$. Thus, $\big(I + (I-\chi_{R_1}(y))*T_{1+} \oast T_{1+}\big)^{-1} \in Y$, implying (\ref{2.121}).

Next, by Lemma \ref{lm2.13} again, we can localize $T_{1+}$ in frequency at each fixed $\eta \in \set R^3$. $T_{1+}$, localized at frequency $\eta$ by convolution, is very close to a product operator, of the type described in Lemma \ref{lemma2.14}: for each $\eta$
\be\lb{lalala}
\lim_{\delta \to 0} \|(e^{i\eta y} \chi_{\delta}(y)) * T_{1+} - \big(\mc F_y T_{1+}(\eta)\big)(x_0, x_1)\, e^{i\eta y} \chi_{\delta}(y)\|_Y = 0.
\ee
Let us prove (\ref{lalala}), while keeping the proof as general as possible. Explicitly written, $(e^{i\eta y} \chi_{\delta}(y)) * T_{1+}$ has the form
$$
(e^{i\eta y} \chi_{\delta}(y)) * T_{1+}(x_0, x_1, y_0) = \int_{\R^3} e^{i\eta(y_0-y)} \chi_\delta(y_0-y) T(x_0, x_1, y) \dd y.
$$
For $s \in \ISO(3)$ and $\|f\|_{L^\infty} \leq 1$, its corresponding contraction $e_{f, 0, s} (e^{i\eta y} \chi_{\delta}(y)) * T_{1+}$ takes the form
$$\begin{aligned}
&e_{f, 0, s} (e^{i\eta y} \chi_{\delta}(y)) * T_{1+}(x_1, y_0) = \\
&= \int_{\R^6} e^{i\eta(y_0+x_0-sx_0-y)} \chi_\delta(y_0+x_0-sx_0-y) f(x_0) T_{1+}(x_0, x_1, y) \dd x_0 \dd y \\
&= \int_{\R^6} e^{i\eta(y_0-y)} \chi_\delta(y_0-y) f(x_0) T_{1+}(x_0, x_1, y+x_0-sx_0) \dd x_0 \dd y \\
&= \int_{\R^3} e^{i\eta(y_0-y)} \chi_\delta(y_0-y) \frak X_{V_f+}(x_1, s^{-1} y + x_1 - s^{-1} x_1) \dd y,
\end{aligned}$$
where, if $\frak X_{V_f+}(x_1, y)$ is given by the measure $g_{\sigma, y}(x)$, then
$$
\frak X_{V_f+}(x_1, s^{-1} y + x_1 - s^{-1} x_1)
$$
is given by the measure $g_{s\sigma, sy}(x)$ and
$$
\|\frak X_{V_f+}(x_1, s^{-1} y + x_1 - s^{-1} x_1)\|_X = \|\frak X_{V_f+}(x_1, y)\|_X \les \|T_{1+}\|_Y.
$$
Due to Lemma \ref{lm2.13}, uniformly in $s$ and $f$, $\|f\|_{L^\infty} \leq 1$,
$$
\|(1-\chi(y/R)) \frak X_{V_f+}(x_1, s^{-1} y + x_1 - s^{-1} x_1)\|_X \to 0
$$
as $R \to \infty$. Then we can fix $R$ such that
$$
\Big\|\int_{\R^3} e^{i\eta(y_0-y)} \chi_\delta(y_0-y) (1-\chi(y/R)) \frak X_{V_f+}(x_1, s^{-1} y + x_1 - s^{-1} x_1) \dd y\Big\|_X < \epsilon.
$$
Within the remaining set of radius $\les R$ we use the fact that
$$
\int_{\R^3} \sup_{y \in \R^3} \big(|\chi_\delta(y_0-y)-\chi_\delta(y_0)| \chi(y/R)\big) \dd y_0 \les \delta (1+R^3).
$$
Then, by definition,
$$\begin{aligned}
&\Big\|\int_{\R^3} e^{i\eta(y_0-y)} (\chi_\delta(y_0-y)-\chi_\delta(y_0)) \chi(y/R) \frak X_{V_f+}(x_1, s^{-1} y + x_1 - s^{-1} x_1) \dd y\Big\|_X \les \\
&\les \int_{\R^3} \bigg\|\int_{\R^3} |\chi_\delta(y_0-y)-\chi_\delta(y_0)| \chi(y/R) |\frak X_{V_f+}(x_1, s^{-1} y + x_1 - s^{-1} x_1)| \dd y\bigg\|_{L^\infty_{x_1}} \dd y_0 \\
&\les \|\frak X_{V_f+}(x_1, s^{-1} y + x_1 - s^{-1} x_1)\|_{L^1_y L^\infty_{x_1}} \int_{\R^3} \sup_{y \in \R^3} \big(|\chi_\delta(y_0-y)-\chi_\delta(y_0)| \chi(y/R)\big) \dd y_0 \\
&\les \|T_{1+}\|_Y \delta (1+R^3).
\end{aligned}$$
Thus, as $\delta$ goes to $0$, uniformly in $s \in \ISO(3)$ and $f$, $\|f\|_{L^\infty} \leq 1$,
\be\lb{diff}
\bigg\|\int_{\R^3} e^{i\eta(y_0-y)} (\chi_\delta(y_0-y)-\chi_\delta(y_0)) \frak X_{V_f+}(x_1, s^{-1} y + x_1 - s^{-1} x_1) \dd y\bigg\|_X \to 0.
\ee
On the other hand,
$$\begin{aligned}
&\int_{\R^3} e^{i\eta(y_0-y)} \chi_\delta(y_0) \frak X_{V_f+}(x_1, s^{-1} y + x_1 - s^{-1} x_1) \dd y = \\
&= \int_{\R^6} e^{i\eta(y_0-y)} \chi_\delta(y_0) f(x_0) T_{1+}(x_0, x_1, y + x_0 - s x_0) \dd x_0 \dd y \\
&= \int_{\R^6} e^{i\eta(y_0+x_0-sx_0-y)} \chi_\delta(y_0+x_0-sx_0) f(x_0) T_{1+}(x_0, x_1, y) \dd x_0 \dd y \\
&= \int_{R^3} e^{i\eta(y_0+x_0-sx_0} \chi_\delta(y_0+x_0-sx_0) f(x_0) \big(\mc F_y T_{1+}(\eta)\big)(x_0, x_1) \dd x_0 \\
&= e_{f, 0, s} \big(\mc F_y T_{1+}(\eta)(x_0, x_1) e^{i\eta y_0} \chi_{\delta}(y_0)\big).
\end{aligned}$$
Thus, (\ref{diff}) precisely means that
$$
\lim_{\delta \to 0} \|(e^{i\eta y} \chi_\delta(y)) \ast T_{1+} - \mc F_y T_{1+}(\eta)\big)(x_0, x_1) e^{i\eta y} \chi_{\delta}(y)\|_Y = 0.
$$

For fixed $\eta_0 \in \R^3$, consider two small radii $\delta_1 > 2 \delta_2$ and cutoff functions $\widehat \chi_1(\eta) := \widehat \chi_{\delta_1}(\eta-\eta_0)$, $\widehat \chi_2(\eta) := \widehat \chi_{\delta_2}(\eta-\eta_0)$. 
Then
\be\begin{aligned}
\chi_2 * (I + T_{1+})^{-1} &= \chi_2 * (I + \chi_1 * T_{1+})^{-1} \\
&= \chi_2 * (I + \chi_1(y) \mc F_y T_{1+}(\eta) - (\chi_1(y) \mc F_y T_{1+}(\eta_0) - \chi_1 * T_{1+}))^{-1}.
\end{aligned}\ee
By Lemma \ref{lemma2.3}, $I + \mc F_y T_{1+}$ is invertible at each point, with uniformly bounded inverses:
\be\begin{aligned}
(I + \mc F_y T_{1+}(\eta))^{-1} &= (I + e^{i x_1 \eta} R_0(|\eta|^2+i0)(x_1, x_0) e^{-ix_0 \eta} V(x_0))^{-1} \\
&= I - e^{i x_1 \eta} R_V(|\eta|^2+i0) V e^{-ix_0 \eta} \\
&:= I - \tilde T_{1+}(\eta) \in \B(L^{\infty}_{x_0}, L^{\infty}_{x_1}).
\end{aligned}\ee
Due to this fact, we can invert these product operators locally. A local inverse is provided by
\be
\chi_2 * (I + \chi_1(y) \mc F_y T_{1+}(\eta)) \oast (I - \chi_1(y) \tilde T_{1+}(\eta)) = \chi_2 * I.
\ee
Note that the norm of $I - \chi_1(y) \tilde T_{1+}(\eta)$ is uniformly bounded by Lemma \ref{lm2.13}, independently of $\delta_2$.

If $\delta_1$ is sufficiently small, then by (\ref{lalala}) the series
\be\begin{aligned}
&\chi_2 * (I + \chi_1(y) \mc F_y T_{1+}(\eta_0) - (\chi_1(y) \mc F_y T_{1+}(\eta_0) - \chi_1 * T_{1+}))^{-1} =\\
&= \sum_{\ell=0}^{\infty} \chi_2 * (I - \chi_1(y) \tilde T_{1+}(\eta)) \big((\chi_1(y) \mc F_y T_{1+}(\eta_0) - e^{i\eta_0y}\chi_1 * T_{1+}) (I - \chi_1(y)\tilde T_{1+}(\eta))\big)^{\ell}
\end{aligned}\ee
converges in $Y$, implying that
\be
\chi_2 * (I+T_{1+})^{-1} \in Y.
\ee

Thus, each point $\eta_0 \in \set R^3$ has some neighborhood $\mc N(\eta_0)$ such that, for some smooth, compactly supported function $\widehat \chi_{\eta_0}$ with $\widehat \chi_{\eta_0}(\eta) = 1$ on $\mc N_{\eta_0}$,
\be
\chi_{\eta_0} * (I + T_{1+})^{-1} \in Y.
\ee
For any ball $B(0, R)$, choose a finite covering
$$
B(0, R) \subset \mc N(\eta_1) \cup \ldots \mc N(\eta_n)
$$
by such neighborhoods and a subordinated partition of unity with $\supp \tilde \chi_{\ell} \subset \mc N(\eta_{\ell})$ and $\sum_{\ell = 1}^n \tilde \chi_\ell(\eta) = 1$ on $B(0, R)$.

Then there exists a smooth, compactly supported function $\tilde \chi_R$ such that $\tilde \chi_R(\eta) = 1$ on $B(0, R)$ and
$$
\tilde \chi_R * (I + T_{1+})^{-1} \in Y.
$$
For large enough $R$, $(I - \chi_{R/2}) * (I + T_{1+})^{-1} \in Y$, so
$$
(I + T_{1+})^{-1} = \tilde \chi_R * (I + T_{1+})^{-1} + (I - \tilde \chi_R) * (I - \chi_{R/2}) * (I + T_{1+})^{-1} \in Y.
$$
Hence $I+T_{1+}$ is invertible in $Y$. By Lemma \ref{lm2.4}, we already knew that $I + T_{1+}$ is invertible in $Z$ and its inverse is $I - T_+$. Since $Y \subset Z$ and the inverse is unique, we obtain that $I - T_+ \in Y$. Hence $W_+ \in X$, by (\ref{eqn2.13}), (\ref{eqn2.29}), and (\ref{2.31}).
%
%
\end{proof}

\subsection{A general formulation of Wiener's theorem}\lb{general_wiener}
Wiener's tauberian theorem, obtained by Wiener in 1932, indicates that if $f \in \widehat {L^1}(\set T)$ is such that $f(\xi) \ne 0$ for all $\xi \in \set T$, then $1/f \in \widehat {L^1}(\set T)$.

We state and apply the theorem in a form adapted to the wave operator problem. We provide this formulation for reference in this section, though it is already contained in the proof of Theorem \ref{theorem_1.1}.

The statement uses of two spaces, $X_A$ and $Y_A$, indexed by a Banach space $A$, $A \subset L^{\infty}$. In the proof $A$ is either $L^{\infty}$ or
\be\lb{aaa}
A := \Big\{g = \int_{S^2} g_{\omega}(x \cdot \omega) \dd \omega\mid \int_{S^2} \|g_{\omega}\|_{\infty} + \|g_{\omega}'\|_{\mc M} \dd \omega < \infty\Big\}.
\ee
Otherwise put, $A$ is the space of integrable combinations of characteristic functions of half-spaces.

In the proof of Theorem \ref{theorem_1.1}, note that $g_{s, y}(x) \in A$ for almost all $y$ and $s$ in (\ref{2.100}) and that (\ref{2.32}) can be strengthened to
\be\lb{2.136}
\int_{\set R^3} \int_{\ISO(3)} \dd \|g_{s, y}\|_A \dd y \les \|V\|_B.
\ee

Redefine elementary transformations to be mappings of the form
$$
e_{g, y, s} f(x) := g(x) f(s x + y),
$$
where $g \in A$, $s \in \ISO(3)$, and $y \in \R^3$. 

Let $X_A$ be the space of two-variable kernels
$$\begin{aligned}
X_A := &\big\{\frak X = \frak X(x, y) \mid \frak X(x, y) = \int_{\ISO(3)} g_{s, y + x - s x}(x) \dd s,\\
&\int_{\set R^3} \int_{\ISO(3)} \dd\|g_{\omega, y}\|_A < \infty\big\},
\end{aligned}$$
where $\dd\|g_{\omega, y}\|_A \in \mc M_{\omega, y}$. Thus $X$ is the space of integrable combinations of elementary transformations.

Since $X_A \subset X$, elements of $X_A$ are well-defined as distributions on $\R^6$.

Let $Y_A$ be the space of three-variable kernels
$$\begin{aligned}
Y_A &:= \{T(x_0, x_1, y) \in Z \mid \forall s \in \ISO(3) T(x_0, x_1, y + x_0 - sx_0) \in Z,\\
&\forall g \in A \forall s \in \ISO(3) \int_{R^3} g(x_0) T(x_0, x, y+x_0-s x_0) \dd x_0 \in X_A\}.
\end{aligned}$$


Thus $Y_A$ consists of those kernels $T$ whose contraction $e_{g, y, s}(T)$ done by means of any elementary transformation $e_{g, y, s}$ belongs to $X_A$.

$g_{s, y}(x) \in \mc M_{s, y, x}$ is an $A$-valued measure, possibly singular in $y$ and $\omega$, such that $\|g_{s, y}\|_A \in \mc M_{s, y}$ is a finite-mass positive Borel measure.


Define the operation $\mc F_y (T_1 \oast T_2) := \mc F_y T_1 \circ \mc F_y T_2$ for any $T_1$, $T_2 \in Y_A$. $(Y_A, \oast)$ is a Banach algebra in this general setting by the same proof as that of Lemma \ref{lemma2.1}.

We state the Wiener-type theorem in $Y_A$.

\begin{theorem}\lb{thm7}
If $T \in Y_A$ is invertible in $Y_A$, then $\widehat T(\eta)$ is invertible in $\B(A, A)$ for every $\eta \in \set R^3$. Conversely, if $\widehat T(\lambda)$ is invertible in $\B(A, A)$ for each $\lambda$, $T = I + L$,~and for some $n \geq 1$
$$
\lim_{y \to 0} \|L(\cdot + y)^n - L^n\|_{Y_A} = 0,\ \lim_{R \to \infty} \sup_{\substack{\|f\|_{L^\infty} \leq 1 \\ s \in \ISO(3)}}\|(1-\chi(y/R)) e_{f, 0, s} L(y)\|_{X_A} =0,
$$
then $T$ is invertible in $Y_A$.
\end{theorem}
The reader is also directed to \cite{bec} and especially to \cite{becgol}, where a similar abstract Wiener theorem is proven in simpler spaces.

In the proof of the main result we used $X_A$ and $Y_A$ for $A = L^{\infty}$. When using Wiener's theorem to invert $I+T_{1+}$ within $Y_A$ for $A$ defined by (\ref{aaa}), we have to prove that the Fourier transform
$$
I + \mc F_y T_{1+}(x_0, x_1, \eta) = I + e^{ix_0\eta} R_0(\eta^2 + i 0)(x_0, x_1) e^{-ix_1\eta} V(x_0)
$$
is invertible for each $\eta$ using Fredholm's alternative. Hence we first show that $\mc F_y T_{1+}(\eta)$ is compact on $A$, then proceed with the proof of Theorem~\ref{thm7}.

\begin{lemma}\lb{TFA} For $V \in B$ and $\eta \in \R^3$, $\mc F_y T_{1+}(x_0, x_1, \eta) \in \B(A_{x_0}, A_{x_1})$ is compact --- with $A$ defined by (\ref{aaa}).
\end{lemma}

\begin{proof}[Proof of Lemma \ref{TFA}]
Recall that by (\ref{2.32}) and (\ref{2.100}) $W_{1+}$ has the form
$$
(W_{1+} f)(x) = \int_{\ISO(3)} \int_{\set R^3} f(s x - y) \dd g_{s, y}(x),
$$
where, for $S_\omega x = x - 2(x \cdot \omega) \omega$,
$$
\dd g_{S_\omega, -y}(x) = \frac 1 2 \delta_{\frac y {|y|}}(\omega) \big(\chi_{(-\infty, |y|/2)}(x \cdot \omega)\ L_1(y) + \chi_{(|y|/2, \infty)}(x \cdot \omega)\ \tilde L_1(y)\big) |y|^{-2} \dd y.
$$
Thus for almost each $y \in \R^3$ $\frak X_{V+}(x, y)$ is in $A_x$ and
$$
\int_{\R^3} \|\frak X_{V+}(x, y)\|_{A_x} \dd y \les \|V\|_B.
$$
Note that the structure of $W_{1+}$ is independent of $V$; only the coefficients in this structure formula depend on $V$. This implies norm continuity, i.e.
$$
\int_{\R^3} \|\frak X_{V_1+}(x, y) - \frak X_{V_2+}(x, y)\|_{A_x} \dd y \les \|V_1 - V_2\|_B.
$$
Therefore the Fourier transform
$$
\mc F_y \frak X(x, \eta) = \int_{\R^3} \frak X_{V+}(x, y) e^{-iy\eta} \dd y
$$
is bounded and continuous for each $\eta$ as a function of $V \in B$ into $A_x$:
$$
\Big\|\int_{\R^3} \big(\frak X_{V_1+}(x, y) - \frak X_{V_2+}(x, y)\big) e^{-iy\eta} \dd y\Big\|_{A_x} \les \|V_1 - V_2\|_B.
$$
This implies that $\mc F_y T_{1+}(x_0, x_1, \eta) \in \B(L^\infty_{x_0}, A_{x_1})$:
$$\begin{aligned}
\|\mc F_y T_{1+}(x_0, x_1, \eta) f(x_0)\|_{A_{X_1}} &= \bigg\|\int_{\R^6} f(x_0) T_{1+}(x_0, x_1, y) e^{-i\eta y} \dd x_0 \dd y\bigg\|_{A_{x_1}} \\
&= \|\mc F_y \frak X_{V_f+}(x, \eta)\|_{A_{x_1}} \les \|V_f\|_B \les \|f\|_{L^\infty} \|V\|_B.
\end{aligned}$$
In particular, for a fixed potential $V \in B$, consider the family of potentials
$$
V_{t, \omega}(x) := \chi_{\{x \mid x \cdot \omega \geq t\}}(x) V(x),
$$
where $t \in \R$, $\omega \in S^2$. Then the map $(t, \omega) \mapsto V_{t, \omega}$ is continuous from $\R \times S^2$ into $B$.

Take $V$ of compact support; then the set of $(t, \omega)$ for which $V_{t, \omega} \ne 0$ is bounded, so the range $\{V_{t, \omega} \mid V_{t, \omega} \ne 0\} \cup \{0\}$ of the mapping $V_{t, \omega}$ is also compact in $B$. Thus for fixed $\eta$ and compact $V$
$$
K = \{\mc F_{\eta} \frak X_{V_{t, \omega}}(x, \eta) \mid t \in \R,\ \omega \in S^2\}
$$
is compact within $A$.

Then consider for $\eta \in \R^3$
$$\begin{aligned}
S &= \Big\{\int_{\R^6} f(x_0) T_{1+}(x_0, x_1, y) e^{-i\eta y} \dd x_0 \dd y \mid \|f\|_A \leq 1\Big\} \\
&= \{\mc F_y \frak X_{V_f+}(x, \eta) \mid \|f\|_A \leq 1\},
\end{aligned}$$
fix $\epsilon>0$, and consider an $\epsilon$-grid $\{e_1, \ldots, e_N\}$ for $K \subset A$.

Every element $f$ in the unit ball of $A$ is a convex combination of characteristic functions of half-planes. Thus each element $\sigma$ of $S$ is a convex combination of elements of $K$, which we approximate by means of the $\epsilon$-grid. Then each $\sigma \in S$ can be approximated to the order of $\epsilon\|\mc F_y T_{1+}(x_0, x_1, \eta)\|_{L^{\infty}_\eta \B(A_{x_0}, A_{x_1})}$ by a convex combination of grid elements belonging to
$$
G = \Big\{\sum_{k=1}^N \alpha_k e_k \mid \sum_{k=1}^N |\alpha_k| \leq 1\Big\},
$$
i.e.\ $\sup_{s \in S} d_A(s, G) < \epsilon$. Since the set $G \subset A$ is compact, it also has a finite $\epsilon$-grid, so $S$ admits a finite $C\epsilon$-grid. As $\epsilon>0$ was chosen arbitrarily, $S$ is compact.

Finally, since $S$ is compact, $\mc F_y T_+(x_0, x_1, \eta)$ is compact in $\B(A_{x_0}, A_{x_1})$ for every $\eta$ when $V \in B$ has compact support. By continuity, this extends to all $V \in B$.
\end{proof}


Now we prove Theorem \ref{thm7}, along the same lines as Theorem \ref{theorem_1.1}.
\begin{proof}[Proof of Theorem \ref{thm7}] One inference is clear: if $T \oast U = U \oast T = I$, then
$$
\widehat T(\lambda) \circ \widehat U(\lambda) = \widehat U(\lambda) \circ \widehat T(\lambda) = \widehat I (\lambda) = I.
$$
To prove the converse, we assume that $\widehat T(\eta)$ is invertible in $\B(A, A)$ for each $\eta$ and explicitly construct the inverse of $T$ in $Y$.

Our strategy is finding an inverse for $T$ on some neighborhood of every point in $\R^3$ and then covering a large ball in $\set R^3$ with finitely many such neighborhoods. What is left is a neighborhood of infinity, which we treat separately.

Let $\chi$ be a smooth cutoff function, such that $\widehat \chi(\eta) = 1$ on some neighborhood of zero and $\widehat \chi$ has compact support, and let
\be
\chi_{\delta}(x) = \delta^{3} \chi(\delta x), \chi_R(x) = R^{3} \chi(R x).
\ee
Then $\widehat{\chi_{\delta}}(\eta) = \widehat{\chi}(\delta^{-1} \eta)$ and $\widehat{\chi_R}(\eta) = \widehat{\chi}(R^{-1} \eta)$.

Firstly, note that
\be
\lim_{R \to \infty} \|\chi_R(y) * L - L\|_Y = 0.
\ee
Also observe that for each fixed $\eta_0$
$$
\lim_{\delta \to 0} \|e^{-i\eta_0 y} \chi_{\delta}(y) * L - \mc F_y L(\eta_0) \chi_{\delta}(y)\|_Y = 0;
$$
see (\ref{lalala}) and its proof.

For each $\eta_0 \in \set R^3$, $\widehat T(\eta_0)$ is invertible; then, for sufficiently small $\delta$, $e^{i y \eta_0}\chi_{\delta}(y) *\widehat T(\lambda) \sim e^{iy\eta_0} \chi_{\delta}(y) \widehat T(\eta_0)$ is also invertible locally.

By choosing a sufficiently large $R$, we can make $(I-\chi_R(y)) * T(y)$ as small as we wish, so we can invert $T$ outside a ball of sufficiently large radius.

We obtain the conclusion after covering $\set R^3$ by finitely many such sets --- one at infinity and several neighborhoods of points in $\R^3$ --- and inverting $T$ on each set.
\end{proof}

This leads directly to the proof of Theorem \ref{thm_general}.
\begin{proof}[Proof of Theorem \ref{thm_general}]
For the $L^p$ boundedness of the wave operators, $L^{\infty}$ is sufficient in the above. However, for a sharper result, it matters that $g_{s, y}$ that appear in (\ref{2.32}) are not general $L^{\infty}$ functions.

In fact, these functions have a specific structure: indeed, in (\ref{2.32})
$$
g_{S_\omega, y}(x) = g(x \cdot \omega), g' \in \mc M.
$$
Making abstraction of $\omega \in S^2$, we retain the fact that $g_{S_\omega, y}(x)$ is essentially a function of one coordinate.

We then define $X_A$ and $Y_A$ and repeat the proof of Theorem \ref{theorem_1.1} with $A$ given by (\ref{aaa}) instead of $A=L^{\infty}$. The Fourier transform $\mc F_y(x_0, x_1, \eta)$ is a compact operator for each $\eta$ by Lemma \ref{TFA}.

Thus Fredholm's alternative applies and $I + \mc F_y(x_0, x_1, \eta)$ is invertible if and only if the spectral condition holds in $A \subset L^{\infty} \subset \langle x \rangle^\sigma L^2$. We refer the reader to Lemma \ref{lem2.1} for more details.

Finally, we need to show that the preconditions (\ref{2.73}) and (\ref{2.74}) are met. However, looking back at the proof given to these two conditions originally in Lemma \ref{lm2.13}, we see that we have actually proved the stronger statements required here.

Finally, the fact that $W_\pm$ and $W_\pm^*$ are in $X_A$ directly implies that they are an integrable combination of isometries composed with multiplication by characteristic functions of half-planes.
\end{proof}

\subsection{Proofs of Corollaries \ref{cor_sobolev}--\ref{wei} and of Propositions \ref{asimptotic} and \ref{multi}}
\begin{proof}[Proof of Corollary \ref{cor_sobolev}]
Let $\dot W^{1, \mc M} = \{f \mid |\dl f| \in \mc M\}$.

Then $\dot W^{1, \mc M}$ has property (\ref{1.12}). Indeed, assume $f \in \dot W^{1, \mc M}$ and let $g(x) = g_{\omega}(x \cdot \omega)$, $g_{\omega} \in \dot W^{1, \mc M}$. We decompose $x \in \set R^d$ into the component perpendicular to $\omega$ and the one parallel to it, $x = (x_{\omega}, x_{\perp})$.

Note that $\partial_{x_{\perp}} (gf) = g \partial_{x_{\perp}} f$, because $g$ is constant in the perpendicular direction. Moreover, 
$$\begin{aligned}
\|\partial_{x_{\omega}} (f g)\|_{\mc M_x} &\leq \|\partial_{x_\omega} f g\|_{\mc M_x} + \|f \partial_{x_\omega} g\|_{\mc M_x} \\
&\leq \|\partial_{x_\omega} f\|_{\mc M_x} \|g\|_{L^\infty_x} + \|f\|_{\mc M_{x_\perp} L^{\infty}_{x_\omega}} \|\partial_{x_\omega} g\|_{L^{\infty}_{x_\perp} \mc M_{x_\omega}} \\
&\les \|f\|_{\dot W^{1, \mc M}} \|g_{\omega}\|_{\dot W^{1, \mc M}}.
\end{aligned}$$
We are using the fact that in $\R$ $\dot W^{1, \mc M} \subset L^{\infty}$.

Thus, $\dot W^{1, \mc M}$ has property (\ref{1.12}) and $W_{\pm}$ and $W_{\pm}^*$ are bounded on this space.

To prove the boundedness of wave operators on $\dot W^{1, 1}$, one has to keep better track of the singular components of the measures involved. To begin with, wave operators take $\dot W^{1, 1} \subset \dot W^{1, \mc M}$ into $\dot W^{1, \mc M}$. Furthermore, note that $T_{1+} \oast T_{1+}$ has the property (\ref{2.73}). This implies that it can be well approximated by convolution with a kernel in the $y$ variable. Then any contraction of $T_{1+} \oast T_{1+}$ can also be well approximated by convolution. Since
$$
T_+ = I - (I+T_{1+})^{-1} = T_{1+} - (I+T_{1+})^{-1} \oast T_{1+}  \oast T_{1+},
$$
it follows that in fact any contraction of $T_+ - T_{1+}$ (in particular $W_+ - W_{1+}$) takes $\dot W^{1, 1}$ to itself --- i.e., beginning with the second power in this formal series, all singularities are sufficiently smoothened out so as not to matter.

It remains to examine the first term $W_{1+}$, which is not well-approximated by convolution. As stated before, this term has the form
$$
(W_{1+}f)(x) = \int_{\R^3} \int_{\ISO(3)} f(sx-y) \dd g_{s, y}(x)
$$
where
$$
dg_{S_\omega, -t\omega}(x) = \frac 1 2\big(\chi_{(-\infty, t/2)}(x \cdot \omega) L_1(t\omega) + \chi_{(t/2, \infty)}(x \cdot \omega) \tilde L_1(t\omega) \big) \dd t \dd\omega.
$$
For each $\omega$ and $t$, the singular part of $\dl (f(S_\omega x-y) g_{S_\omega, -t\omega})$ is concentrated on the plane $\{x \mid x \cdot \omega = t/2\}$ and is precisely given by the jump discontinuity
$$
\frac 1 2 (L_1(t\omega) - \tilde L_1(t\omega)) \delta_{t/2}(x \cdot \omega) f(S_\omega x+t\omega) \omega. 
$$
This is a singular vector-valued measure supported on $\{x \mid x \cdot \omega = t/2\}$. However, integrating the singular part of $|\dl (f(S_\omega x-y) g_{S_\omega, -t\omega})|$ for fixed $\omega$ over all $t \geq 0$ we obtain the measure
$$
\frac 1 2 |L_1(2x\cdot \omega) - \tilde L_1(2x \cdot \omega)| \chi_{[0, \infty)}(x \cdot \omega) |f(x)|.
$$
This is no longer a singular measure, since its support is a whole half-plane and we show it is in fact absolutely continuous. Again decompose $x$ into $x_\perp$ perpendicular to $\omega$ and $x_\omega$ parallel to $\omega$. Then $\|f(x)\|_{L^1_{x_\perp} L^\infty_{x_\omega}} \les \|f\|_{\dot W^{1, 1}}$ and
$$
\|L_1(2x\cdot \omega) - \tilde L_1(2x \cdot \omega)\|_{L^\infty_{x_\perp} L^1_{x_\omega}} \les \|L_1(t\omega)\|_{L^1_t} + \|\tilde L_1(t\omega)\|_{L^1_t} < \infty
$$
for $g$-almost all $\omega$, so
$$
\||L_1(2x\cdot \omega) - \tilde L_1(2x \cdot \omega)| \chi_{[0, \infty)}(x \cdot \omega) |f(x)|\|_{L^1_x} \les \|f\|_{\dot W^{1, 1}} (\|L_1(t\omega)\|_{L^1_t} + \|\tilde L_1(t\omega)\|_{L^1_t}).
$$
Integrating in $\omega$ as well, we obtain a finite quantity. Thus $\dl (W_{1+}f)(x)$ has no singular part, so $W_{1+}f \in \dot W^{1, 1}$ --- hence $W_+$ takes values in $\dot W^{1, 1}$ as well.

Next, consider a function $f \in \dot H^s$, $0 \leq s \leq 1/2$. Note that
$$
\|f\|_{\dot H^s_x} \sim \|f\|_{L^2_{x_\perp} \dot H^s_{x_\omega}} + \|f\|_{L^2_{x_\omega} \dot H^s_{x_\perp}}.
$$
Since $g(x) = g_\omega(x\cdot \omega)$ is constant in the perpendicular directions,
\be\lb{1.17}
\|fg\|_{L^2_{x_\omega} \dot H^s_{x_\perp}} \leq \|f\|_{L^2_{x_\omega} \dot H^s_{x_\perp}} \|g_{\omega}\|_{L^{\infty}}.
\ee
Likewise,
\be\lb{1.18}
\|fg\|_{L^2_{x_\perp} \dot H^s_{x_\omega}} \les \|f\|_{L^2_{x_\perp} \dot H^s_{x_\omega}} \|g\|_{L^{\infty}_{x_\perp} (\dot W^{1, 1} \cap L^{\infty})_{x_\omega}}.
\ee
This is based on the one-dimensional Leibniz rule, valid for $0 \leq s < 1/2$,
\be\lb{2.147}
\|f_1 f_2\|_{\dot H^s} \leq \|f_1\|_{\dot H^s} \|f_2\|_{\dot W^{1, \mc M}}.
\ee
To prove (\ref{2.147}), start from the fact that
\be\lb{2.148}
\|\chi_{[x_0, \infty)}(x) f\|_{\dot H^s} \les \|f\|_{\dot H^s}
\ee
for $0 \leq s < 1/2$ and use Minkowski's inequality. In turn, (\ref{2.148}) follows from the theory of the Hilbert transform with $A_p$-weights, since $|x|^s$ is an $A_2$-weight for $-1 < s < 1$.

Then $W_\pm$ and $W_\pm^*$ are bounded on $\dot H^s$, $-1/2 < s < 1/2$.

When $s \geq 1/2$, $\dot H^s$ functions on $\set R$ are Lipschitz continuous or in $BMO$, so this estimate can no longer hold.

For other Sobolev spaces $W^{s, p}$, we establish that
$$
M(\xi) = \frac{|\xi|^s}{|\xi_{\omega}|^s + |\xi_{\perp}|^s}
$$
is a bounded Fourier multiplier on $L^p$ (obvious when $p = 2$). Indeed, for the symbol $|\xi|^s$, one has that
$$
\big|\partial^m |\xi|^s\big| \les_m |\xi|^{s-m}
$$
and likewise for $|\xi_{\omega}|^s$ and $|\xi_{\perp}|^s$. It follows that $\big|\partial^m M(\xi)\big| \les_m |\xi|^{-m}$.

Thus, $M(\xi)$ is a Mihlin multiplier and is bounded on $L^p$, $1<p<\infty$. We use this fact in the computation
$$\begin{aligned}
\|f\|_{\dot W^{s, p}} &= \||\dl|^s f\|_p \\
&= \|M(\dl)(|\dl_{\perp}|^s f + |\dl_{\omega}|^s f)\|_p \\
&\leq \|M(\dl)\|_{\B(L^p, L^p)} (\|f\|_{L^p_{\omega} \dot W^{s, p}_{\perp}} + \|f\|_{L^p_{\perp} \dot W^{s, p}_{\omega}}).
\end{aligned}$$
Both of these anisotropic norms are bounded as in (\ref{1.17}--\ref{1.18}), as long as $1<p<\infty$ and $0 \leq sp < 1$. One uses a one-dimensional Leibniz-type inequality akin to (\ref{2.147}), namely
$$
\|f_1 f_2\|_{\dot W^{s, p}} \leq \|f_1\|_{\dot W^{s, p}} \|f_2\|_{\dot W^{1, \mc M}}.
$$
Thus
$$
\|fg\|_{\dot W^{s, p}} \les \|f\|_{\dot W^{s, p}} \|g_{\omega}\|_{\dot W^{1, \mc M}}.
$$
At the endpoint $sp=1$, this inequality is again false because $\dot W^{s, p} \subset BMO$, which is not preserved by sharp cutoff functions.

Then $W_\pm$ and $W_\pm^*$ are bounded on $W^{s, p}$, $|s|<1/p$, by Theorem \ref{thm_general}.

Finally, $W_\pm$ and $W_\pm^*$ are bounded on $C_b$. Indeed, we approximate the potential $V$ in $B$ by a sequence of potentials $V_n \in B \cap L^{3/2+\epsilon}$.

The wave operators $W_+^n$ corresponding to $V_n$ preserve $W^{2, 3/2+\epsilon}$, by Corollary \ref{cor_sobolev_mare}.

Take $f \in C_b$ and approximate it in $L^{\infty}$ by $f_m \in W^{2, 3/2+\epsilon}$. Since $W_+^n f_m \in W^{2, 3/2+\epsilon} \subset C_b$, for fixed $n$, one has that $W_+^n f \in C_b$ for each $n$. As $n \to \infty$, $W_+^n f \to W_+ f$ in $L^{\infty}$, so $W_+ f \in C_b$. The same is true of $W_-$~and~$W_\pm^*$.

Since $W_\pm$ and $W_\pm^*$ are bounded on $C_b$ for each $V_n$, we retrieve the same conclusion in the limit for $V$.
\end{proof}

\begin{proof}[Proof of Corollary \ref{cor_sobolev_mare}] With no loss of generality, let $s\geq 0$ and assume that $-1$ is not an eigenvalue of $H$. Following \cite{yajima} and \cite{yajima9}, by (\ref{1.4})
$$
H P_c W_{\pm} = W_{\pm} H_0,\ (H + I)^{-1} P_c W_{\pm} = W_{\pm} (H_0 + I)^{-1}.
$$
$f \in W^{s, p}$ is equivalent to $(H_0 + 1)^{s/2} f \in L^p$, thus to $(H P_c + 1)^{s/2} W_{\pm} f \in L^p$, provided that $V \in B$.

We prove this implies $(H_0 + 1)^{s/2} W_{\pm} f \in L^p$ as well, by induction.

Since $V \in W^{s, 3/(s+2)} \subset L^{3/2}$, the eigenstates of $H$ are in $W^{s+2, q}$ for any $q \in (1, \infty)$. Indeed, assume $\lambda<0$ is an eigenvalue of $H$ and let $f \in L^2$, $f \ne 0$, be such that $(-\Delta + V - \lambda) f = 0$. Then
$$
\|\dl f\|_2^2 + \langle V f, f \rangle - \lambda \|f\|_2^2 = 0.
$$
Write $V = V_1 + V_2$, $V_1 \in L^{\infty}$, $\|V_2\|_{L^{3/2}} << 1$. We obtain that $f \in H^1 \subset L^6$ and $V f \in L^{6/5}$, so $f \in W^{2, 6/5}$ and
$$
f + R_0(\lambda) V f = 0.
$$
Iterating, we obtain that $V f \in W^{s, 6/5}$, so $f \in W^{s+2, 6/5}$. Splitting $V$ into $V = V_1 + V_2$, $V_1 \in \mc S$, $\|V_2\|_{W^{s, 3/(s+2)}} << 1$, we write
$$
f = - (I + R_0(\lambda) V_2)^{-1} V_1 f.
$$
Then $f \in W^{s+2, q}$ for any $q \in (1, \infty)$.

Because $W_{\pm} f \in P_c W^{s, p}$ and $P_c W^{s, p} \subset W^{s, p}$,
$$
\|(H P_c + 1)^{s/2} W_{\pm} f \|_{L^p} = \|(H+1)^{s/2} W_{\pm} f \|_{L^p}.
$$
The conclusion reduces to
$$
\|(H + 1)^{s/2} W_{\pm} f \|_{L^p} \sim \|(H_0+1)^{s/2} W_{\pm} f \|_{L^p}.
$$
It suffices to show that when $0 \leq \sigma \leq \max(0, s-2)$
\be\lb{equiv}
\|(H + 1) g\|_{W^{\sigma, p}} \sim \|(H_0 + 1) g\|_{W^{\sigma, p}}.
\ee
Taking into account the identities
$$
(H + 1) = (I + V R_0(1)) (H_0 + 1),\ H_0 + 1= (I + V R_0(1))^{-1} (H+1),
$$
we need to show that $(I + V R_0(1))$ is bounded and invertible on $W^{\sigma, p}$. Thus, it is required of $V$ that for $0 \leq \sigma \leq \min (0, s-2)$
\be\lb{leib}
\|V g\|_{W^{\sigma, p}} \les \|g\|_{ W^{\sigma+2, p}}.
\ee

When $\sigma=0$, we assume that $V \in L^{3/2}$ if $p<3/2$, $V \in L^{3/2+\epsilon}$ for $p=3/2$, and $V \in L^p$ if $p>3/2$. The Sobolev embedding then implies (\ref{leib}).

When $\sigma>0$, we take $V \in W^{\sigma, 3/(\sigma+2)} \subset L^{3/2}$ for $(\sigma+2) p <3$, $V \in W^{\sigma, \epsilon + 3/(\sigma+2)}$ if $(\sigma+2) p =3$, and $V \in W^{\sigma, p}$ if $(\sigma+2) p > 3$; then we obtain (\ref{leib}) by the fractional Leibniz rule, see \cite{taylor}.

Thus $I + V R_0(1)$ is bounded on $W^{\sigma, p}$. Furthermore, by approximating $V$ with smoother potentials we obtain that $I + V R_0(1)$ is compact on $W^{\sigma, p}$. By Fredholm's alternative either it is invertible or the equation
$$
(I + V R_0(1)) f = 0
$$
has a nonzero solution $f \in W^{\sigma, p}$. However, this would imply that $-1$ is an eigenvalue of $H$, which we assumed was not the case.

Concerning $W_\pm^*$, note that
$$
H_0 W_{\pm}^* = W_{\pm} H P_c.
$$
$f \in W^{s, p}$ implies $(H P_c + 1)^{s/2} f \in L^p$, so $(H_0 + 1)^{s/2} W_{\pm}^* f \in L^p$. The general conclusion again follows by interpolation.
\end{proof}

\begin{proof}[Proof of Corollary \ref{wei}] By Lemma \ref{lema2.13}, $T_{1+} \in Y_{\beta}$; see (\ref{eq2.84}). $Y_{\beta}$ is translation-invariant in the sense that
$$
\|T(x_0, x_1, y + y_0)\|_{Y_{\beta}} \les \langle y_0 \rangle^{\beta} \|T(x_0, x_1, y)\|_{Y_{\beta}}.
$$
We then repeat the proof of the main theorem in this subalgebra and conclude that, when $V \in \langle x \rangle^{-\alpha} L^2$, the wave operators are in the space $X_{\beta}$.

$X_{\beta}$ is characterized by a stronger decay at infinity; its elements are integrable combinations of operators of the form
$$
(Tf)(x) = g_{s, y}(x) f(sx-y)
$$
with
$$
\int_{\ISO(3)} \int_{\set R^3} \langle y \rangle^{\beta} \dd \|g_{s, y}\|_{L^{\infty}_x} < \infty.
$$
Symmetries and convolution with $\langle y \rangle^{-\beta} L^1_y$ preserve $\langle y \rangle^{-\beta} L^p_y$, $1 \leq p \leq \infty$. We see this by interpolating between $p=1$ and $p=\infty$, where it is clear.
\end{proof}

\begin{proof}[Proof of Corollary \ref{multi}] Let $P_c f_1 = W_+ g_1$, $P_c g = W_+ g_2$. By the intertwining property (\ref{1.4})
$$
e^{itH} P_c f_1 = W_+ e^{itH_0} g_1,\ e^{itH} P_c f_2 = W_+ e^{itH_0} g_2.
$$
(\ref{multilin}) then becomes
$$
\Big|\int_0^{\infty} U (W_+ e^{itH_0} g_1) (W_+ e^{itH_0} g_2) \dd t\Big| \les \|g_1\|_{H^{-1}} \|g_2\|_{H^{-1}} \|U\|_{L^{\infty} \cap L^{3/2}}.
$$
Using the structure formula Theorem \ref{theorem_1.1}, we write $W_+-I$ as an integrable combination of elementary transformations of the form
$$
e_{g, y, s} f(x) = g(x) f(s x - y).
$$
By Minkowski's inequality, it suffices to prove the result when each $W_+$ is replaced by an elementary transformation. Translations commute with $e^{itH_0}$ and the scalar function $g(x) \in L^{\infty}$ can be made part of $U$.

Thus the statement reduces to the same estimate for the free evolution:
\be\lb{est_free}
\Big|\int_0^{\infty} \tilde U\, e^{itH_0} g_1\, e^{itH_0} g_2 \dd t\Big| \les \|g_1\|_{H^{-1}} \|g_2\|_{H^{-1}} \|U\|_{L^{\infty} \cap L^{3/2}}.
\ee

Write (\ref{est_free}) in stationary form, using the Fourier transform, as
\be\lb{janpet}
\Big|\int \widehat U(\xi_1 - \xi_2) \frac {\widehat f_1(-\xi_1) \widehat f_2(\xi_2)}{|\xi_1|^2 + |\xi_2|^2} \dd \xi_1 \dd \xi_2 \Big| \les \|f_1\|_{H^{-1}} \|f_2\|_{H^{-1}} \|U\|_{L^{\infty} \cap L^{3/2}}.
\ee
The space of Schur multipliers $M(A, B) \subset L^{\infty}(A \times B)$ is
\be\lb{schur}\begin{aligned}
&M(A, B) := \Big\{h(x, y) \in L^{\infty}(A \times B) \mid \exists \mu \in \mc M\ s.t.\\
&h(x, y) = \int f_\mu(x) g_\mu(y) \dd \mu,\ \int \|f_\mu\|_{L^{\infty}_A} \|g_\mu\|_{L^{\infty}_B} \dd \mu < \infty\Big\}.
\end{aligned}\ee
By the results of Janson--Peetre \cite{peetre}, $\frac {|\xi_1| |\xi_2|} {|\xi_1|^2 + |\xi_2|^2}$ is a Schur multiplier.

Then we replace $\frac 1 {|\xi_1|^2 + |\xi_2|^2}$ by $\frac 1 {|\xi_1| |\xi_2|}$ in (\ref{janpet}), which becomes
$$
\Big|\int \widehat U(\xi_1 - \xi_2) \frac {\widehat f_1(-\xi_1) \widehat f_2(\xi_2)}{|\xi_1| |\xi_2|} \dd \xi_1 \dd \xi_2 \Big| \les \|f_1\|_{H^{-1}} \|f_2\|_{H^{-1}} \|U\|_{L^{\infty} \cap L^{3/2}}.
$$
This is equivalent to
$$
\int \widehat U(\xi_1 - \xi_2) \widehat f_1(-\xi_1) \widehat f_2(\xi_2) \dd \xi_1 \dd \xi_2 \les \|f_1\|_{\dot H^1 + L^2} \|f_2\|_{\dot H^1 + L^2} \|U\|_{L^{\infty} \cap L^{3/2}},
$$
which is true by Young's and H\"{o}lder's inequalities.
\end{proof}

\begin{proof}[Proof of Proposition \ref{asimptotic}]
By (\ref{2.81tert}), if $V \in \langle x \rangle^{-3/2-\epsilon}L^2$, then $T_{1+}$ is in $Y_{1+\epsilon}$ as defined by the decay condition (\ref{cond_asimp}).

Applying the proof of Theorem \ref{theorem_1.1} in $Y_{1+\epsilon}$, we obtain that $T_+ \in Y_{1+\epsilon}$, hence $W_+ \in X_{1+\epsilon}$.

Due to the fact that
$$
T_+ = (I + T_{1+})^{-1} \oast T_{1+},
$$
we obtain that $W_+-I$ is a combination of operators of the form $\frak X_{V_f+}(x, s^{-1}y+x-s^{-1}x)$, see (\ref{xfrak}). In all of these operators, the leading term $g^1_{s, y}(x)$ is constant in $x$ for $g$-almost all $s$ and $y$, since it has the explicit form given by (\ref{2.81tertt}).
\end{proof}

\subsection{Proof of Theorem \ref{thm_s}}
In deriving the results, we use a more general representation formula along the lines of Lemma \ref{lemma2.1}:
\begin{lemma}\lb{reprezentare}
Assume that for any $f$, $g \in \mc S$
$$\begin{aligned}
\langle W f, g \rangle &=\lim_{\epsilon \to 0} \int_\R \langle W(t) e^{-itH_0} e^{-\epsilon|t|} f, g \rangle \dd t =\int_0^{\infty} \langle \widehat W(\lambda) R_{0a}(\lambda) f, g \rangle \dd \lambda.
\end{aligned}$$
Let $T$ be defined by $\mc F_{x_0, x_1, y} T(\xi_0, \xi_1, \eta) = \mc F_{a, b}\big(\widehat W(|\eta|^2)\big)(\xi_0+\eta, \xi_1+\eta)$ or, equivalently, $\mc F_y T(x_0, x_1, \eta) = e^{-ix_0\eta} \widehat W(|\eta|^2)(x_0, x_1) e^{ix_1\eta}$. Then
$$
\langle W f, g \rangle = \int T(x_0, x_1, y) f(x-y) \ov g(x) \dd x_0 \dd y \dd x.
$$
\end{lemma}
\begin{proof} A computation shows that
$$\begin{aligned}
\langle W f, g \rangle &= \int_0^{\infty} \langle \widehat W(\lambda) R_{0a}(\lambda) f, g \rangle \dd \lambda \\
&= \int_{\R^6} \mc F_{a, b}\big(\widehat W(|\eta_0|^2)\big)(\eta_0, \eta_1) \widehat f(\eta_0) \ov {\widehat g}(\eta_1) \dd \eta_1 \dd \eta_0 \\
&= \int_{\R^6} \mc F_{a, b}\big(\widehat W(|\eta|^2)\big)(\eta, \eta+\xi) \widehat f(\eta) \ov{\widehat g}(\eta+\xi) \dd \eta \dd \xi.
\end{aligned}$$
\end{proof}

We next derive a formula for the scattering operator $S$ given by (\ref{1.5}).
\begin{lemma} Assume that $V \in L^{3/2, 1}$ and zero is neither an eigenvalue, nor a resonance for $H = -\Delta+V$. Then
\be\lb{1.15}
S = \frac 1 {2\pi}\int_\R (I + R_0(\lambda+i0) V) R_{Va}(\lambda) (I+ V R_0(\lambda+i0)) \dd \lambda.
\ee
Moreover, let
\be\begin{aligned}\lb{2.152}
T_S &= I + (I + T_{1-})^{-1} \oast (T_{1+} - T_{1-}).
\end{aligned}\ee
Then
$$
\langle S f, g \rangle = \int_{\R^9} T_S(x_0, x, y) f(x-y) \ov g(x) \dd x_0 \dd y \dd x.
$$
\end{lemma}
More generally, this is still true when $V \in L^{3/2, \infty}_0$. The second formula allows us to apply our formalism to this problem.
\begin{proof} By (\ref{eqn1.6}) and (\ref{eqn1.7}),
$$\begin{aligned}
S 
&= \Big(P_c + i \int_{-\infty}^0 e^{itH_0} V e^{-itH} P_c \dd t\Big)\Big(P_c + i \int_0^{\infty} e^{itH} P_c V e^{-itH_0} \dd t\Big) \\
&= P_c - \int_0^{\infty} \int_{-\infty}^0 e^{-itH_0} V e^{i(t-s)H} P_c V e^{isH_0} \dd s \dd t + \\
&+ i \int_0^{\infty} (e^{itH} P_c V e^{-itH_0} + e^{-itH_0} V e^{itH} P_c) \dd t\\
&= \frac 1 {2\pi} \int_0^{\infty} \big(I + R_0(\lambda+i0) V\big) R_{Va}(\lambda) \big(I + V R_0(\lambda+i0)\big) \dd \lambda.
\end{aligned}$$
We made use of the fact that
$$
P_c = \int_0^{\infty} R_{Va}(\lambda) \dd \lambda,
$$
which follows from the Spectral Theorem, see \cite{reesim1}, and
$$\begin{aligned}
\big(\chi_{[0, \infty)}(t) e^{itH_0}\big)^{\wedge}(\lambda-i0) &= i R_0(\lambda-i0);\ (e^{itH})^{\wedge}(\lambda) = R_{Va}(\lambda); \\
\big(\chi_{(-\infty, 0]}(t) e^{itH_0}\big)^{\wedge}(\lambda+i0) &= -iR_0(\lambda+i0).
\end{aligned}$$

On $(-\infty, 0)$ the integral is a sum of Dirac measures:
$$
\sum_{\ell=1}^N (I + R_0(\lambda_\ell) V)  \langle \cdot, f_\ell \rangle f_\ell (I+ V R_0(\lambda_\ell)) = 0.
$$
Here $f_\ell$ are the eigenfunctions corresponding to the eigenvalues $\lambda_\ell$ of $H$, so $(I + R_0(\lambda_\ell) V) f_\ell = 0$.

By the resolvent identity $R_V + R_V V R_0 = R_0$ we further get
$$\begin{aligned}
S&=\int \big(I + R_0(\lambda+i0) V\big) R_{Va}(\lambda) \big(I + V R_0(\lambda+i0)\big) \dd \lambda \\
&= \int i \big(I + R_0(\lambda+i0) V\big) R_V(\lambda-i0) V R_0(\lambda+i0) - i R_0(\lambda+i0) + \\
&+ i \big(I + R_0(\lambda+i0) V\big) R_V(\lambda-i0) \dd \lambda \\
&= \int (-1)\big(I + R_0(\lambda+i0) V\big) R_V(\lambda-i0) V R_{0a}(\lambda) + \\
&+ R_{0a}(\lambda) + R_0(\lambda+i0) V R_{0a}(\lambda) \dd \lambda.
\end{aligned}$$

Due to Lemma \ref{reprezentare}, the resolvent identity $A^{-1}-B^{-1} = A^{-1} (B-A) B^{-1}$, and the fact that $T_+ = T_{1+} \oast (I + T_{1+})^{-1}$, $T_S$ is given by
$$\begin{aligned}
I + T_{1+} - T_- \oast (I+T_{1+}) &= I + (T_+ - T_-) \oast (I + T_{1+}) \\
&= I + (I+T_-)^{-1} \oast (T_{1+}-T_{1-}).
\end{aligned}$$

\end{proof}

With the representation (\ref{2.152}), the proof of Theorem \ref{thm_s} is straightforward.
\begin{proof}[Proof of Theorem \ref{thm_s}] (\ref{2.152}) implies that $T_S \in Y$ when $V \in B$, so $S$ is in $X$, hence it has a representation of the form (\ref{eqn1.8}).

To show that this is actually of the form (\ref{eqn1.11}), we identify a left ideal within $Y$ of elements whose contractions have this form, then prove that $T_S$ is in this ideal.

Namely, let $X_0 \subset X$ be the subset of those elements for which $g_{s, y}(x)$ is constant for $g$-almost every $s$ and $y$ and let $Y_0 \subset Y$ be the subset of elements $T$ whose every contraction $eT$ is in $X_0$.

Note that $Y_0 \subset Y$ is a left ideal and $T_{1+} - T_{1-}$ belongs to it. Indeed, for a given potential $V$, if $\frak X_{V \pm}$ represent $W_{1\pm}$ as per (\ref{xfrak}), then
$$
(\frak X_{V+} - \frak X_{V-}) f(x) = \int_{S^2} \int_{[0, \infty)} (K_{1+}(x, t\omega)-K_{1-}(x, t\omega)) f(x+t\omega) \dd t \dd \omega,
$$
where by (\ref{eqn2.90})
$$
K_{1+}(x, t\omega)-K_{1-}(x, t\omega) = \int_\R \widehat V(s\omega) e^{-its/2} e^{is\omega \cdot x} s \dd s = (L_1 + \tilde L_1)((t-2x\cdot \omega) \omega).
$$
Thus, for $S_\omega x = x - 2(x \cdot \omega) \omega$,
$$
(\frak X_{V+} - \frak X_{V-}) f(x) = \int_{\R^3} \int_{\ISO(3)} f(s x+y) \dd g_{s, y},
$$
where $\dd g_{S_\omega, -y} = \delta_{\frac y {|y|}}(\omega) (L_1 + \tilde L_1)(y) |y|^{-2} \dd y$ does not depend on $x$.

Then $\frak X_{V+} - \frak X_{V-} \in X_0$. Every contraction of $T_{1+} - T_{1-}$ has the form $\frak X_{V_f+}(x, s^{-1}y+x-s^{-1}x) - \frak X_{V_f-}(x, s^{-1}y+x-s^{-1}x)$, so is in $X_0$ by the same argument.

Thus, $T_{1+} - T_{1-} \in Y_0$. Formula (\ref{2.152}) shows that $T_S \in Y_0$, so $S \in X_0$. 
\end{proof}

\section*{Acknowledgments} I thank Professor Kenji Yajima for his insightful comments on this paper, as well as the anonymous referee for several useful observations.

\end{document}